\documentclass[a4paper,fleqn]{cas-sc}

\usepackage[numbers]{natbib}
\usepackage{cleveref}
\usepackage{ushort}
\usepackage{amsmath}
\usepackage{amssymb}
\usepackage{multirow}
\usepackage{amsfonts}
\usepackage{amsthm}
\usepackage{algorithm}
\usepackage{algpseudocode}
\algrenewcommand\algorithmicrequire{\textbf{Input:}}
\algrenewcommand\algorithmicensure{\textbf{Output:}}
\usepackage{siunitx}
\usepackage{subcaption}

\newcommand{\vertiii}[1]{{\left\vert\kern-0.25ex\left\vert\kern-0.25ex\left\vert #1 
		\right\vert\kern-0.25ex\right\vert\kern-0.25ex\right\vert}}

\newcommand{\inner}[2]{\left(#1, #2\right)}
\newcommand{\dginner}[2]{a\left(#1, #2\right)}
\newcommand{\llnorm}[1]{\big\lVert #1\big\rVert_{L_2(\Omega)}}

\newcommand{\ilnorm}[1]{\left\lVert #1\right\rVert_{L_\infty(0,T;L_2(\Omega))}}
\newcommand{\Llnorm}[1]{\left\lVert #1\right\rVert_{L_2(0,T;L_2(\Omega))}}

\newcommand{\enorm}[1]{\left\lVert #1\right\rVert_{V}}
\newcommand{\ienorm}[1]{\left\lVert #1\right\rVert_{L_\infty(0,T;V)}}

\newcommand{\lenorm}[1]{\left\lVert #1\right\rVert_{L_2(0,T;V)}}
\newcommand{\gnorm}[1]{\left\lVert #1\right\rVert_{L_2(\Gamma_N)}}

\newcommand{\ignorm}[1]{\left\lVert #1\right\rVert_{L_\infty(0,T;L_2(\Gamma_N))}}
\newcommand{\lgnorm}[1]{\left\lVert #1\right\rVert_{L_2(0,T;L_2(\Gamma_N))}}

\newcommand{\ginner}[2]{\left(#1, #2\right)_{L_2(\Gamma_N)}}
\newcommand{\Nphi}{N_\varphi}

\newcommand{\ndof}{n_{\text{dof}}}

\newcommand{\dgelliptic}{\boldsymbol R}

\newcommand{\Dgv}{\boldsymbol{H}^s(\mathcal{E}_h)}
\newcommand{\Dgvh}{\boldsymbol{\mathcal{D}}_k(\mathcal{E}_h)}
\newcommand{\jump}[2]{J^{\alpha_0,\beta_0}_0\left(#1,#2\right)}
\newcommand{\njump}[1]{[#1\otimes\boldsymbol{n}_e]}

\newcommand{\sume}{\sum\limits_{e\subset\Gamma_h\cup\Gamma_D}\int_e}

   \def\g{\gamma}

\def\q{\quad} \def\qq{\quad\quad}
 
\newcommand{\bs}[1]{\boldsymbol{#1}}
\numberwithin{equation}{section}

\newtheorem{theorem}{Theorem}[section]

\newtheorem{lemma}{Lemma}[section]
\newtheorem{rmk}{Remark}

\begin{document}
\let\WriteBookmarks\relax
\def\floatpagepagefraction{1}
\def\textpagefraction{.001}

\shorttitle{DG-POD for viscoelasticity}    

\shortauthors{B.Shin and Y.Jang}  

\title[mode = title]{A stable DG-POD reduced-order method for dynamic linear viscoelasticity}  



%
\author{Byeong-Chun Shin}[orcid=0000-0001-8625-5597]
\ead{bcshin@chonnam.ac.kr}
\credit{Investigation, Methodology, Writing - original draft, review and editing}
\author{Yongseok Jang}[orcid=0000-0002-2036-558X]
\cormark[1]

\ead{yongseok.jang@chonnam.ac.kr}
\credit{Formal analysis, Investigation, Methodology, Software, Visualization,
Writing - original draft, review and editing}

\affiliation{organization={Department of Mathematics, Chonnam National University},
            addressline={77 Yongbong-ro, Buk-gu}, 
            city={Gwangju},
            postcode={61186},
            country={South Korea}}






\cortext[1]{Corresponding author}



\begin{abstract}
We present a reduced-order modeling framework based on proper orthogonal decomposition for dynamic linear viscoelasticity governed by generalized Maxwell models. The hereditary constitutive law is reformulated using internal variables, and the resulting problem is discretized in space by a symmetric interior penalty discontinuous Galerkin method and in time by either the backward Euler or Crank--Nicolson scheme. The reduced model is obtained by projecting the fully discrete system onto low-dimensional spaces constructed from displacement snapshots. For the full-order discontinuous Galerkin formulation, we show well-posedness and derive an \textit{a priori} stability bound without using Gronwall's inequality. Consequently, the bound grows at most linearly, rather than exponentially, with the final time. We further derive spatial error estimates and an error decomposition that separates the effects of spatial discretization, temporal discretization, and proper orthogonal decomposition truncation. The reduced formulation replaces the large full-order systems with low-dimensional problems, enabling efficient long-time simulations. Numerical experiments verify the predicted convergence behavior, assess different inner products for constructing the reduced basis, and demonstrate accurate recovery of transient oscillations and viscoelastic relaxation with substantial computational savings, even after accounting for snapshot generation and basis construction.
\end{abstract}





\maketitle

\section{Introduction}\label{sec:intro}
Linear viscoelasticity provides a fundamental framework for describing the dynamic behavior of materials exhibiting both elastic and viscous responses, such as polymers, biological tissues, elastomers, and composite materials. Unlike purely elastic solids, viscoelastic materials exhibit stress relaxation and energy dissipation through hereditary constitutive relations, in which the current stress depends on the entire deformation history. Such models have found widespread applications in structural dynamics, vibration analysis, biomedical engineering, and soft-material design \cite{koeller1984applications,lin2009viscoelastic}. In recent years, generalized Maxwell models based on sums of exponential relaxation functions have attracted considerable attention because they can accurately represent a broad class of viscoelastic materials, while more general relaxation kernels, including power-law kernels, can be efficiently approximated by finite sums of exponentials \cite{xiao2016equivalence}.

The numerical simulation of dynamic viscoelastic problems remains computationally demanding primarily due to the hereditary nature of the constitutive law. The resulting memory effects require the storage and repeated evaluation of the complete solution history, leading to significant computational overhead. Even when the constitutive equation is reformulated using internal variables \cite{johnson} for exponential relaxation kernels, the resulting full-order models still involve large-scale systems of equations after spatial discretization, making simulations computationally intensive for large-scale problems.

Among high-order discretization techniques, discontinuous Galerkin (DG) methods provide several attractive features for transient wave propagation problems, including local conservation, geometric flexibility, high-order accuracy, hp-adaptivity, and excellent parallel scalability \cite{DG,houston2002discontinuous,wihler2006locking}. These advantages have motivated their successful application to dynamic viscoelasticity, including both classical and fractional-order constitutive models \cite{jang2023priori,jang2024discontinuous}. However, the increased number of degrees of freedom associated with DG discretizations significantly increases the computational cost, especially for long-time simulations and many-query applications.

Projection-based reduced-order modeling has emerged as one of the most effective approaches for reducing the computational complexity of large-scale dynamical systems. Among various techniques, the proper orthogonal decomposition (POD) method constructs low-dimensional approximation spaces from representative solution snapshots and has been successfully applied to a broad range of time-dependent partial differential equations because of its simplicity, robustness, and computational efficiency. In particular, the method of snapshots introduced by Sirovich in 1987 provided a practical and efficient framework for computing POD bases from data, significantly advancing its application to fluid dynamics and related fields \cite{Si1, Si2, Si3}. In general, the proper orthogonal decomposition (POD) constructs an orthonormal basis that optimally represents a collection of snapshots by minimizing the projection error in the least-squares sense \cite{kunisch,kunisch2,liang,SS2}. Numerous studies have demonstrated the effectiveness of POD for fluid dynamics, structural mechanics, and coupled multiphysics problems. Nevertheless, comparatively little attention has been devoted to POD reduced-order models for DG discretizations of second-order viscoelastic wave equations formulated using internal variables. Moreover, the influence of different POD inner products on the resulting reduced approximation has not been systematically investigated in the DG framework.

The objective of this work is to develop a POD-based reduced-order model for discontinuous Galerkin discretizations of second-order viscoelastic wave equations. The proposed approach combines internal-variable formulations with projection-based model reduction to efficiently handle memory effects while reducing computational complexity. Also, we investigate the impact of different POD inner products on the accuracy of the reduced model within the DG framework.

The novelty of this work lies in a reduced-order computational framework that combines an internal-variable reformulation, a displacement-based solution strategy, and proper orthogonal decomposition within a symmetric interior penalty discontinuous Galerkin discretization of dynamic linear viscoelasticity. The internal variables replace the hereditary convolution by local evolution equations, thereby avoiding repeated quadrature over the complete deformation history. At the fully discrete level, the coupled problem is reorganized so that only the displacement system is solved at each time step, while the velocity and internal variables are updated algebraically. The displacement system is then projected onto a low-dimensional space, while retaining the mass, stiffness, jump-stabilization, and internal-variable operators of the full-order formulation. We further establish well-posedness and derive an \textit{a priori} stability bound without using Gr\"onwall's inequality; the resulting estimate grows at most linearly, rather than exponentially, with the final time and is therefore suitable for long-time simulations. In addition, we derive an error decomposition that separates the spatial discretization, temporal discretization, and proper orthogonal decomposition truncation contributions in terms of the mesh size, time step, and neglected singular values. Numerical experiments show that the reduced model reproduces both transient oscillations and viscoelastic relaxation with accuracy comparable to the full-order approximation, while greatly reducing the linear-system dimension and achieving large speed-ups in the solution phase.

The remainder of this paper is organized as follows. After introducing the notation and preliminary results, \cref{sec:model} presents the viscoelastic model and its internal-variable formulation. \Cref{sec:dg} introduces the DG discretization and its analysis, while \cref{sec:pod} develops the POD reduced-order model. \Cref{sec:numeric} contains the numerical experiments, and \cref{sec:conclusion} concludes the paper.

\paragraph{Preliminary}
We adopt standard notation throughout. The spaces $L_p(\Omega)$, $H^s(\Omega)$, and $W^s_p(\Omega)$
(with $s,p \ge 0$) denote the usual Lebesgue and Sobolev spaces.
For a normed space $X$, the symbol $\lVert\cdot\rVert_X$ denotes the norm in $X$, which for inner-product spaces is induced by the corresponding inner product. In particular,
$\lVert\cdot\rVert_{L_2(\Omega)}$ denotes the norm generated by the $L_2(\Omega)$ inner product,
written for brevity as $(\cdot,\cdot)$. When integration is restricted to a subset
$S \subset \bar{\Omega}$ we write $(\cdot,\cdot)_{L_2(S)}$.

For time-dependent functions we use the Bochner space notation. Given a normed space $X$,
the space $L_p(0,T;X)$ consists of measurable functions $f:(0,T)\to X$ with norm
\[
\lVert f\rVert_{L_p(0,T;X)}=\left(\int_0^T \lVert f(t)\rVert_X^p\,dt\right)^{1/p},\qquad 1\le p<\infty.
\]
For $p=\infty$ the norm is defined by the essential supremum
\[
\lVert f\rVert_{L_\infty(0,T;X)}=\mathrm{ess\,sup}_{0\le t\le T}\lVert f(t)\rVert_X .
\]
When convenient, the upper limit $T$ will be replaced by another time $t \in [0,T]$.

For vector- and tensor-valued functions we use the same notation for inner products.
For example,
\[(\boldsymbol{v},\boldsymbol{w})=\int_\Omega \boldsymbol{v}\cdot\boldsymbol{w}\,d\Omega,
\qquad
(\ushort{\boldsymbol{v}},\ushort{\boldsymbol{w}})
=\int_\Omega \ushort{\boldsymbol{v}}:\ushort{\boldsymbol{w}}\,d\Omega
=\sum_{i,j=1}^d\int_\Omega v_{ij}w_{ij}\,d\Omega,
\]
for vector fields $\boldsymbol{v},\boldsymbol{w}$ and second-order tensors
$\ushort{\boldsymbol{v}},\ushort{\boldsymbol{w}}$. The outer product of two vectors $\boldsymbol{v},\boldsymbol{w}\in\mathbb{R}^d$ is denoted by
$\boldsymbol{v}\otimes\boldsymbol{w}$ and defined componentwise as
\[
(\boldsymbol{v}\otimes\boldsymbol{w})_{ij}=v_i w_j,\qquad i,j=1,\ldots,d .
\]
Boldface notation is also used for vector-valued spaces, e.g.
\[
\boldsymbol L_2(\Omega)=[L_2(\Omega)]^d,
\qquad
\boldsymbol H^s(\Omega)=[H^s(\Omega)]^d,
\]
with analogous definitions for other Sobolev and finite element spaces.

\section{Model problem}\label{sec:model}
Let $\Omega \subset \mathbb{R}^d$ denote the interior of a bounded polytopic domain occupied by a linear, homogeneous, and isotropic viscoelastic solid \cite{VE}. We consider the evolution of the displacement field $\boldsymbol{u}$ and the Cauchy stress tensor $\ushort{\boldsymbol{\sigma}}$ over the time interval $(0,T]$, where $T>0$ is a prescribed final time. The motion of the material is governed by the balance of linear momentum,
\begin{equation}
\rho \ddot{\boldsymbol{u}} - \nabla \cdot \ushort{\boldsymbol{\sigma}} = \boldsymbol{f}_b
\qquad \text{on } \Omega \times (0,T],
\label{eq:primal:visco}
\end{equation}
where $\rho$ denotes the constant mass density and $\boldsymbol{f}_b$ represents an applied body force. Throughout, $\dot{\boldsymbol{u}}$ and $\ddot{\boldsymbol{u}}$ represent the velocity and acceleration, respectively.

To complete the boundary-value formulation, the boundary $\partial\Omega$ is decomposed into two disjoint parts $\Gamma_D$ and $\Gamma_N$, where $\Gamma_D$ is assumed to have positive surface measure. The system is then subject to homogeneous displacement conditions and prescribed surface tractions,
\begin{alignat}{2}
\boldsymbol{u}(t) &= \boldsymbol{0} &\qquad& \text{on } \Gamma_D \times [0,T], \label{bc:Dirichlet} \\
\ushort{\boldsymbol{\sigma}}(t)\cdot\boldsymbol{n} &= \boldsymbol{g}_N(t) && \text{on } \Gamma_N \times [0,T], \label{bc:Neumann}
\end{alignat}
where $\boldsymbol{n}$ denotes the outward unit normal vector defined almost everywhere on $\Gamma_N$, and $\boldsymbol{g}_N$ represents the prescribed Neumann data. The initial configuration of the solid is specified by
\begin{equation}
\boldsymbol{u}(0)=\boldsymbol{u}_0,
\qquad
\dot{\boldsymbol{u}}(0)=\boldsymbol{w}_0 .
\label{ic}
\end{equation}

The formulation is completed by introducing a constitutive relation linking the stress to the deformation history. For linear viscoelastic materials this relation is commonly expressed through a Volterra-type hereditary integral, reflecting the fading-memory property of the material \cite{hunter1976mechanics,golden2013boundary}. Accordingly, the stress tensor is given by
\begin{equation}
\ushort{\boldsymbol{\sigma}}(t)
=\ushort{\boldsymbol{D}}\varphi(t)\ushort{\boldsymbol{\varepsilon}}(0)
+\int_0^t\ushort{\boldsymbol{D}}\varphi(t-s)\ushort{\dot{\boldsymbol{\varepsilon}}}(s)\,ds,
\label{eq:consti:visco}
\end{equation}
where $\ushort{\boldsymbol{D}}$ denotes the fourth-order elasticity tensor, assumed positive definite and satisfying the standard minor and major symmetries
\[
D_{ijkl}=D_{jikl}=D_{ijlk}=D_{klij}.
\]
The infinitesimal strain tensor $\ushort{\boldsymbol{\varepsilon}}$ is defined as the symmetric gradient of the displacement field,
\[
\varepsilon_{ij}(\boldsymbol{v})
=\frac{1}{2}\left(
\frac{\partial v_i}{\partial x_j}
+\frac{\partial v_j}{\partial x_i}\right),
\qquad i,j=1,\ldots,d,
\]
and we use the shorthand notation $\ushort{\boldsymbol{\varepsilon}}(t)=\ushort{\boldsymbol{\varepsilon}}(\boldsymbol{u}(t))$ in \eqref{eq:consti:visco}.

The scalar function $\varphi(t)$ characterizes the stress relaxation behaviour and determines the specific viscoelastic material model (see \cite{drozdov1998viscoelastic,findley2013creep,golden2013boundary}). In this work we consider the \emph{generalized Maxwell solid}, for which the relaxation function is represented by a Prony series,
\begin{equation}
\varphi(t)
=\varphi_0
+\sum_{q=1}^{N_\varphi}
\varphi_q e^{-t/\tau_q}.
\label{eq:stressrelax:exponential}
\end{equation}
Here $N_\varphi\in\mathbb{N}$, $\{\tau_q\}_{q=1}^{N_\varphi}$ are strictly positive relaxation times, and $\{\varphi_q\}_{q=0}^{N_\varphi}$ are positive coefficients satisfying the normalization $\varphi(0)=\sum_{q=0}^{\Nphi}\varphi_q=1$.

\subsection{Internal variables}\label{ssec:internal}
Direct evaluation of the Volterra integral in \eqref{eq:consti:visco} requires the complete velocity history to be stored at every spatial point, resulting in high memory usage and computational cost. To avoid this difficulty, we exploit the Prony series representation \eqref{eq:stressrelax:exponential} to obtain a formulation that is local in time. This is achieved by introducing a set of auxiliary internal variables $\{\ushort{\boldsymbol{\zeta}}_q(t)\}_{q=1}^{N_\varphi}$ associated with the individual relaxation mechanisms of the generalized Maxwell solid. 

Specifically, the $q$-th internal variable is defined through the temporal convolution
\begin{equation}
\ushort{\boldsymbol{\zeta}}_q(t)
=\int_0^t\varphi_q e^{-(t-s)/\tau_q}\dot{\boldsymbol{u}}(s)\,ds,
\label{eq:internal_var_def}
\end{equation}
which represents the contribution of the $q$-th relaxation mode to the viscoelastic response. Differentiating \eqref{eq:internal_var_def} with respect to time and applying Leibniz's rule yields the local evolution equation
\begin{equation}
{\tau_q}\ushort{\dot{\boldsymbol{\zeta}}}_q(t)
+\ushort{\boldsymbol{\zeta}}_q(t)
={\tau_q}\varphi_q \dot{\boldsymbol{u}}(t),
\qquad q=1,\dots,N_\varphi,
\label{eq:internal_var_ode}
\end{equation}
subject to the initial condition $\ushort{\boldsymbol{\zeta}}_q(0)=\boldsymbol{0}$.

Since the strain tensor is defined through the symmetric gradient operator, which acts only on the spatial variables, it commutes with the time integration appearing in \eqref{eq:internal_var_def}. Consequently, the stress tensor in \eqref{eq:consti:visco} can be expressed directly in terms of the spatial derivatives of these internal variables. This reformulation replaces the nonlocal history dependence of the Volterra integral by a system of local evolution equations, thereby yielding a Markovian structure that is more amenable to numerical approximation.

Substituting the internal variable representation into the constitutive relation \eqref{eq:consti:visco} and exploiting the Prony series form \eqref{eq:stressrelax:exponential}, the momentum balance equation \eqref{eq:primal:visco} can be expressed in terms of the displacement field and the internal variables. In particular, the governing equation becomes
\begin{equation}
\rho \ddot{\boldsymbol{u}}-\varphi_0\nabla\cdot\ushort{\boldsymbol{D}}\ushort{\boldsymbol{\varepsilon}}(\boldsymbol{u})
-\sum_{q=1}^{N_\varphi}\nabla\cdot\ushort{\boldsymbol{D}}
\ushort{\boldsymbol{\varepsilon}}(\ushort{\boldsymbol{\zeta}}_q)
=\boldsymbol{f}
\label{eq:visco_internal}
\end{equation}
where $\boldsymbol{f}:=\boldsymbol{f}_b+\sum_{q=1}^{N_\varphi} \varphi_qe^{-t/\tau_q}\nabla\cdot\ushort{{\boldsymbol {D}}}\ushort{{\boldsymbol\varepsilon}}(\boldsymbol u_0)$ and  the internal variables $\{\ushort{\boldsymbol{\zeta}}_q\}_{q=1}^{N_\varphi}$ satisfy the evolution equations \eqref{eq:internal_var_ode}. Consequently, the original integro-differential formulation is replaced by a coupled system consisting of the displacement equation \eqref{eq:visco_internal} and the internal-variable evolution equations \eqref{eq:internal_var_ode}. This representation eliminates the global history dependence of the Volterra integral and yields a system that is local in time, which is more suitable for numerical discretization.

\begin{rmk}[Another form of internal variables]
Integration by parts in time in \eqref{eq:consti:visco} yields the equivalent representation
\[
\ushort{\boldsymbol{\sigma}}(t)
=\ushort{\boldsymbol{D}}\ushort{\boldsymbol{\varepsilon}}(t)
+\int_0^t\ushort{\boldsymbol{D}}\dot{\varphi}(t-s)\ushort{\boldsymbol{\varepsilon}}(s)\,ds .
\]
Analogously to \eqref{eq:internal_var_def}, one may introduce an alternative set of internal variables by replacing the convolution term above. This formulation corresponds to the \textit{displacement form} introduced in \cite{jang2023priori}, whereas the definition in \eqref{eq:internal_var_def} is commonly referred to as the \textit{velocity form}. The two formulations are mathematically equivalent; however, the velocity form adopted here is often more convenient computationally, particularly in applications such as viscoelastic fluid mechanics. We refer to \cite{jang2023priori} for further details.
\end{rmk}

\section{Discontinuous Galerkin method}\label{sec:dg}
We now introduce the discontinuous Galerkin discretization of the viscoelastic system \eqref{eq:visco_internal}. To this end, we first define the computational mesh and the associated discontinuous finite element spaces.

We briefly summarize the discontinuous Galerkin framework and refer to \cite{DG} for further details.  
Let $\Omega$ be partitioned into a finite collection of closed elements $E$, where each $E$ is a triangle in two dimensions or a tetrahedron in three dimensions. The intersection of two distinct elements is either empty or consists of a common vertex, edge, or face.
The diameter of an element $E$ is defined by
\[
h_E := \sup_{x,y\in E}\lVert x-y\rVert ,
\]
where $\lVert\cdot\rVert$ denotes the Euclidean norm, and $|E|$ denotes the measure of $E$.  
Similarly, if $e\subset\partial E$ is an edge (in 2D) or a face (in 3D), we denote its measure by $|e|$.

Let $\mathcal{E}_h$ be the set of all elements and define the mesh size
\[
h:=\max_{E\in\mathcal{E}_h} h_E .
\]
We assume the partition is quasi-uniform, meaning that there exists a constant $C>0$ such that
\[
h \le C h_E \qquad \text{for all } E\in\mathcal{E}_h .
\]

Let $\Gamma_h$ denote the set of all interior edges (in 2D) or faces (in 3D).  
For each $e\in\Gamma_h$ we define a unit normal vector $\boldsymbol n_e$.  
If $e\subset\partial\Omega$, then $\boldsymbol n_e$ is the outward unit normal.  
For an interior edge $e\subset E_i\cap E_j$ with $i<j$, the normal $\boldsymbol n_e$ is taken to point from $E_i$ to $E_j$.

We then introduce the broken Sobolev space
\[
H^s(\mathcal{E}_h)
=\left\{
v\in L_2(\Omega) ~ \big|~ v|_E \in H^s(E)
\ \text{for all } E\in\mathcal{E}_h
\right\},
\]
equipped with the norm
\[
\vertiii{v}_{H^s(\mathcal{E}_h)}
=\left(
\sum_{E\in\mathcal{E}_h}
\lVert v\rVert_{H^s(E)}^2
\right)^{1/2}.
\]
Note that $H^s(\Omega)\subset H^s(\mathcal{E}_h)$ and
$H^{s+1}(\mathcal{E}_h)\subset H^s(\mathcal{E}_h)$.
These definitions extend to vector-valued functions.

For an element $E\subset\mathbb{R}^d$ we denote by $\mathcal{P}_k(E)$ the space of polynomials of degree at most $k$,
\[
\mathcal{P}_k(E)
=\mathrm{span}
\left\{
x_1^{i_1}\cdots x_d^{i_d} ~ \big| ~ 
\sum_{m=1}^d i_m \le k
\right\}.
\]
The DG finite element space is then defined by
\[
\mathcal{D}_k(\mathcal{E}_h)
=\left\{
v\in H^1(\mathcal{E}_h)
~ \big| ~ 
v|_E\in\mathcal{P}_k(E)
\ \text{for all }E\in\mathcal{E}_h
\right\}.
\]
The corresponding vector-valued space is defined componentwise.

Let two elements $E_i^e$ and $E_j^e$ share a common edge or face $e$, with $i<j$.  
For a vector field $\boldsymbol v$ and a second-order tensor $\ushort{\boldsymbol v}$ defined on $E_i^e$ and $E_j^e$, we define the average and jump across $e$ by
\[
\{\boldsymbol v\}
=\frac{(\boldsymbol v|_{E_i^e})|_e + (\boldsymbol v|_{E_j^e})|_e}{2},
\qquad
\{\ushort{\boldsymbol v}\}
=\frac{(\ushort{\boldsymbol v}|_{E_i^e})|_e + (\ushort{\boldsymbol v}|_{E_j^e})|_e}{2},
\]
\[
[\boldsymbol v]
=(\boldsymbol v|_{E_i^e})|_e
-(\boldsymbol v|_{E_j^e})|_e,
\qquad
[\boldsymbol v\otimes\boldsymbol n_e]
=(\boldsymbol v|_{E_i^e})|_e\otimes\boldsymbol n_e
-(\boldsymbol v|_{E_j^e})|_e\otimes\boldsymbol n_e .
\]
If $e\subset\partial\Omega$, we set
\[
\{\boldsymbol v\}=\boldsymbol v|_e,
\qquad
[\boldsymbol v]=\boldsymbol v|_e\cdot\boldsymbol n_e,
\qquad
[\boldsymbol v\otimes\boldsymbol n_e]
=
\boldsymbol v|_e\otimes\boldsymbol n_e .
\]

Finally, we define the jump penalty operator
\[
\jump{\boldsymbol v}{\boldsymbol w}
=\sum_{e\subset\Gamma_h\cup\Gamma_D}
\frac{\alpha_0}{|e|^{\beta_0}}
\int_e[\boldsymbol v]\cdot[\boldsymbol w]
\,de ,
\]
where $\alpha_0$ and $\beta_0$ are positive constants.

\subsection{DG bilinear form}\label{ssec:dg}
Using the definitions of jump and average, we can define the DG bilinear form
$a:\Dgv\times\Dgv\to\mathbb{R}$, for $s>3/2$, by
\begin{align}
a(\boldsymbol v,\boldsymbol w)
=&
\sum_{E\in\mathcal{E}_h}\int_E
\ushort{\boldsymbol D}\ushort{\boldsymbol\varepsilon}(\boldsymbol v)
:\ushort{\boldsymbol\varepsilon}(\boldsymbol w)\,dE
-\sum_{e\subset\Gamma_h\cup\Gamma_D}\int_e
\{\ushort{\boldsymbol D}\ushort{\boldsymbol\varepsilon}(\boldsymbol v)\}
:[\boldsymbol w\otimes\boldsymbol n_e]\,de
\nonumber\\
&-\sum_{e\subset\Gamma_h\cup\Gamma_D}\int_e
\{\ushort{\boldsymbol D}\ushort{\boldsymbol\varepsilon}(\boldsymbol w)\}
:[\boldsymbol v\otimes\boldsymbol n_e]\,de
+ J_0^{\alpha_0,\beta_0}(\boldsymbol v,\boldsymbol w),
\label{eq:DG_biform}
\end{align}
for all $\boldsymbol v,\boldsymbol w\in\Dgv$.

We define the associated DG energy norm
\[
\enorm{\boldsymbol v}
=\left(
\sum_{E\in\mathcal{E}_h}
\int_E
\ushort{\boldsymbol D}\ushort{\boldsymbol\varepsilon}(\boldsymbol v)
:\ushort{\boldsymbol\varepsilon}(\boldsymbol v)\,dE
+\jump{\boldsymbol v}{\boldsymbol v}
\right)^{1/2},
\qquad
\boldsymbol v\in\Dgv .
\]
A direct comparison of the definitions yields
\begin{align}
a(\boldsymbol v,\boldsymbol v)
=\enorm{\boldsymbol v}^2
-2\sum_{e\subset\Gamma_h\cup\Gamma_D}
\int_e\{\ushort{\boldsymbol D}\ushort{\boldsymbol\varepsilon}(\boldsymbol v)\}:\njump{\boldsymbol v}\,de .
\label{eq:relation:dgnorm}
\end{align}

Note that using inverse polynomial trace inequality \cite{WARBURTON20032765}, we can obtain bounds for the interior penalty terms. For example, if $\alpha_0$ is sufficiently large and $\beta_0(d-1)\geq1$, we have
\begin{align}
    \sume\bigg|\{\ushort{\boldsymbol D}\ushort{\boldsymbol \varepsilon}(\boldsymbol  v)\}:\njump{\boldsymbol{w}}de\bigg|
    \leq\frac{C}{\sqrt{\alpha_0}}\left(\enorm{\boldsymbol{v}}^2+\jump{\boldsymbol{w}}{\boldsymbol{w}}\right),\label{eq:vector:dg:bdd2}
\end{align}
where $C$ is a positive constant independent of $\boldsymbol{v},\boldsymbol{w}\in\Dgvh$. Consequently, for large enough penalty parameters, the DG bilinear form is coercive and continuous. In other words,
there exist positive constants $\kappa$ and $K$ such that
\begin{align*}
    \kappa\enorm{\boldsymbol{v}}^2\leq\dginner{\boldsymbol{v}}{\boldsymbol{v}},\qquad\text{and}\qquad\left|\dginner{\boldsymbol{v}}{\boldsymbol{w}}\right|\leq K\enorm{\boldsymbol{v}}\enorm{\boldsymbol{w}},\qquad\forall \boldsymbol{v},\boldsymbol{w}\in\Dgvh,
\end{align*}
where $\kappa$ and $K$ are independent of $\boldsymbol{v}$ and $\boldsymbol{w}$.

\subsection{Variational formulation}\label{ssec:weakform}
We are now ready to state the weak formulation of the viscoelastic problem.
Let $\boldsymbol v\in\Dgv$ be a test function. Multiplying the momentum
balance equation \eqref{eq:visco_internal} by $\boldsymbol v$, integrating
over $\Omega$ and adding a jump penalty term for velocity give
the weak formulation of the viscoelastic stated as:
find $\boldsymbol u_h(t)\in\Dgvh$ and
$\{\boldsymbol \zeta_{hq}(t)\}_{q=1}^{N_\varphi}\subset\Dgvh$ such that for all
$\boldsymbol v\in\Dgv$ and $t\in(0,T]$,
\begin{alignat}{2}
\inner{\rho\ddot{\boldsymbol u}_h}{\boldsymbol v}
+\varphi_0
\dginner{\boldsymbol u_h}{\boldsymbol v}
+\sum_{q=1}^{N_\varphi}
\dginner{\boldsymbol \zeta_{hq}}{\boldsymbol v}+\jump{\dot{\boldsymbol{u}}_h}{\boldsymbol{v}}
= F(v),&
\label{eq:weak1}\\
{\tau_{q}}\dot{\boldsymbol \zeta}_{hq}
+\boldsymbol \zeta_{hq}
={\tau_q}\varphi_q\dot{\boldsymbol u}_h,&
\qquad q=1,\ldots,N_\varphi .
\label{eq:weak2}
\end{alignat}
where $F(v)=(\boldsymbol f_0,\boldsymbol v)+\ginner{\boldsymbol g_N}{\boldsymbol v}.$

Let $\boldsymbol{\phi}_i$ be a global basis of $\Dgvh$ such that
\[\Dgvh=\mathrm{span}\{\boldsymbol{\phi}_1,\ldots,\boldsymbol{\phi}_{\ndof}\},\]
where $\ndof$ denotes degrees of freedom. Using the bases, we can write the weak solution as
\begin{equation}
\boldsymbol{u}_h(\boldsymbol{x},t)=\sum_{i=1}^{\ndof}\mathfrak{u}_i(t)\boldsymbol{\phi}_i(\boldsymbol{x}).\label{eq:u:dof}
\end{equation}
In this manner, we can also have, for each $q$,
    \begin{equation}
    \boldsymbol{\zeta}_{hq}(\boldsymbol{x},t)=\sum_{i=1}^{\ndof}z_{qi}(t)\boldsymbol{\phi}_i(\boldsymbol{x}).\label{eq:zeta:dof}
\end{equation}
Also, if we define a mass matrix $M$, a stiffness matrix $A$, a jump matrix $J$ by
\[
M_{ij}=\inner{\boldsymbol{\phi_j}}{\boldsymbol{\phi_i}},\qquad A_{ij}=\dginner{\boldsymbol{\phi_j}}{\boldsymbol{\phi_i}},\qquad\text{and}\qquad J_{ij}=\jump{\boldsymbol{\phi_j}}{\boldsymbol{\phi_i}},\qquad\text{for }i,j=1,\ldots,\ndof,
\]
solving \eqref{eq:weak1} is equivalent to solving the following ODE system:
\begin{equation}
    \rho M\ddot{\boldsymbol{\mathfrak{u}}}(t)+\varphi_0A{\boldsymbol{\mathfrak{u}}}(t)+\sum_{q=1}^{\Nphi}A\boldsymbol{z}_q(t)+J\dot{\boldsymbol{\mathfrak{u}}}(t)=\boldsymbol{b}(t),\label{eq:semi:linear}
\end{equation}
where $\boldsymbol{\mathfrak{u}}(t)=({\mathfrak{u}_i(t)})_{i=1}^{\ndof}$, $\boldsymbol{z}_q(t)=(z_{qi}(t))_{i=1}^{\ndof}$ for each $q$, and $\boldsymbol{b}(t):=(F(t;\boldsymbol{\phi}_i))_{i=1}^{\ndof}$. Also, \eqref{eq:zeta:dof} can be rewritten by
\begin{equation}
    \tau_q \dot{\bs{z}}_q(t)+\bs z_q(t)=\tau_q\varphi_q\dot{\bs{\mathfrak{u}}}(t)\label{eq:semi:zeta},
\end{equation}
in vector form with respect to unknown DOFs. To impose the initial condition, we solve
\begin{equation}
M\dot{\boldsymbol{\mathfrak u}}(0)={\boldsymbol{\mathfrak w}}_0\qq \text{and}\qq A{\boldsymbol{\mathfrak u}}(0)={\boldsymbol{\mathfrak u}}_0,\label{eq:semi:ic}
\end{equation}
where $(\boldsymbol{\mathfrak w}_0)_i=\inner{\bs w_0}{\bs\phi_i}$ and $(\boldsymbol{\mathfrak u}_0)_i=\dginner{\bs u_0}{\bs\phi_i}$. Obviously, we have $\bs z_q(0)=\bs 0$, $\forall q$.

\begin{theorem}
The mass matrix $M$ is symmetric positive definite. Moreover, if $\alpha_0>0$ is sufficiently large and $\beta_0(d-1)\ge 1$, then the SIPG stiffness matrix $A$ is symmetric positive definite and the jump matrix $J$ is symmetric positive semidefinite.
\end{theorem}

\begin{proof}
For any $\boldsymbol{x}\in\mathbb{R}^{\ndof}$, set
$\boldsymbol{v}_h=\sum\limits_{i=1}^{\ndof}x_i\boldsymbol{\phi}_i$.
Then
$
\boldsymbol{x}^\top M\boldsymbol{x}
=
(\boldsymbol{v}_h,\boldsymbol{v}_h)
=
\lVert \boldsymbol{v}_h\rVert_{L_2(\Omega)}^2$.
Since $\{\boldsymbol{\phi}_i\}_{i=1}^{\ndof}$ is a basis, $\boldsymbol{v}_h=\bs 0$ implies $\boldsymbol{x}=\bs 0$. Hence $M$ is symmetric positive definite.
Similarly,
\[
\boldsymbol{x}^\top A\boldsymbol{x}
=
\dginner{\boldsymbol{v}_h}{\boldsymbol{v}_h},\q \text{and}\q \boldsymbol{x}^\top J\boldsymbol{x}
=
\jump{\boldsymbol{v}_h}{\boldsymbol{v}_h}
.
\]
By coercivity of the SIPG bilinear form, for large $\alpha_0$ and $\beta_0$,
$A$ is symmetric positive definite and therefore invertible. However, $J$ is symmetric positive semidefinite. In general $J$ has a nontrivial kernel, since any conforming function with zero jump satisfies $\jump{\boldsymbol{v}_h}{\boldsymbol{v}_h}=0$. Hence $J$ is not invertible in general.
\end{proof}

To obtain a first-order dynamical system, we introduce the velocity variable $\boldsymbol{\mathfrak w}(t)
=
\dot{\boldsymbol{\mathfrak u}}(t)$.
Since $M$ is invertible, 
the system \eqref{eq:semi:linear}-\eqref{eq:semi:zeta} can be expressed compactly as
\begin{equation}
\dot{\boldsymbol y}(t)
=
\mathcal A\boldsymbol y(t)
+
\boldsymbol g(t),
\label{eq:firstorder:compact}
\end{equation}
where $I,O\in\mathbb R^{\ndof\times \ndof}$
denote the identity and zero matrices, respectively, and
\[
\boldsymbol y
=
\begin{bmatrix}
\boldsymbol{\mathfrak u}\\
\boldsymbol{\mathfrak w}\\
\boldsymbol z_1\\
\vdots\\
\boldsymbol z_{N_\varphi}
\end{bmatrix}, \
\boldsymbol g
=
\begin{bmatrix}
\boldsymbol 0\\
\rho^{-1}M^{-1}\boldsymbol b\\
\boldsymbol 0\\
\vdots\\
\boldsymbol 0
\end{bmatrix},
\text{ and }
\mathcal A
=
\begin{bmatrix}
O & I & O & \cdots & O
\\
-\rho^{-1}\varphi_0M^{-1}A
&
-\rho^{-1}M^{-1}J
&
-\rho^{-1}M^{-1}A
&
\cdots
&
-\rho^{-1}M^{-1}A
\\
O
&
\varphi_1 I
&
-\tau_1^{-1}I
&
&
O
\\
\vdots
&
\vdots
&
&
\ddots
&
\vdots
\\
O
&
\varphi_{N_\varphi} I
&
O
&
\cdots
&
-\tau_{N_\varphi}^{-1}I
\end{bmatrix}.
\]

\begin{theorem}[Well-posedness of the semi-discrete system]
Let
$\boldsymbol b\in C([0,T];\mathbb R^{\ndof})$. Then, for any initial data
and
$\boldsymbol z_q(0)=\boldsymbol 0$, $q=1,\ldots,N_\varphi$, the semi-discrete system
\eqref{eq:semi:linear}-\eqref{eq:semi:zeta} admits a unique solution on $[0,T]$.
\end{theorem}

\begin{proof}
Since $M$ is invertible, we can derive
the first-order system
\[
\dot{\boldsymbol y}(t)
=
\mathcal A\boldsymbol y(t)
+
\boldsymbol g(t),
\]
where $\mathcal A\in\mathbb R^{(N_\varphi+2)\ndof\times (N_\varphi+2)\ndof}$ is a constant matrix and
$\boldsymbol g\in C([0,T];\mathbb R^{(N_\varphi+2)\ndof})$. The right-hand side is globally Lipschitz continuous with respect to $\boldsymbol y$ and continuous in time. Therefore the Picard-Lindel\"of theorem implies the existence and uniqueness of a solution on $[0,T]$.
\end{proof}

The preceding results guarantee the existence and uniqueness of the semi-discrete solution. We next turn to the analysis of its approximation properties. The objective is to derive \textit{a priori} bounds for the discretization error and to establish convergence rates in terms of the mesh size $h$.

\subsection{\textit{A priori} analysis}\label{ssec:apriori}
We follow the standard error-splitting argument used for discontinuous Galerkin approximations of viscoelastic problems; see, for example, \cite{jang2023priori,jang2024discontinuous}. We first collect the auxiliary estimates needed in the subsequent \textit{a priori} analysis. Hereafter, we assume sufficient regularity of the data:
\[
\boldsymbol f\in L_2(0,T;\boldsymbol L_2(\Omega)),\quad
\boldsymbol g_N\in C^1(0,T;\boldsymbol L_2(\Gamma_N)),\quad
\boldsymbol u_0\in \boldsymbol H^s(\Omega),\quad
\boldsymbol w_0\in \boldsymbol L_2(\Omega).
\]

\begin{lemma}
Assume that $\alpha_0>0$ and $\beta_0(d-1)\ge 1$. Then the following estimates hold.
For any $\boldsymbol v\in \boldsymbol H^1(\mathcal E_h)$, there exists a constant $C>0$, independent of $\boldsymbol v$ and $h$, such that
\begin{equation}
\sum_{E\in\mathcal E_h}
\lVert \nabla \boldsymbol v\rVert_{L_2(E)}^2
\le
C\enorm{\boldsymbol v}^2 .
\label{eq:kordDG}
\end{equation}
Moreover, for any $\boldsymbol v,\boldsymbol w\in\Dgvh$ and for any interior face $e$ shared by two elements $E_1,E_2\in\mathcal E_h$,
\begin{align}
\left|
\int_e
\{\ushort{\boldsymbol D}\ushort{\boldsymbol\varepsilon}(\boldsymbol v)\}
:
\njump{\boldsymbol w}\,de
\right|
\le
\frac{C}{\sqrt{\alpha_0}}
\bigg(
&
\lVert
\ushort{\boldsymbol D}\ushort{\boldsymbol\varepsilon}(\boldsymbol v)
\rVert_{L_2(E_1)}^2
+
\lVert
\ushort{\boldsymbol D}\ushort{\boldsymbol\varepsilon}(\boldsymbol v)
\rVert_{L_2(E_2)}^2
+
\frac{\alpha_0}{|e|^{\beta_0}}
\lVert[\boldsymbol w]\rVert_{L_2(e)}^2
\bigg).
\label{eq:vector:dg:bdd1}
\end{align}
\end{lemma}
The estimate \eqref{eq:kordDG} follows from the discrete Korn inequality
\cite{DG,brenner2004korn}, while \eqref{eq:vector:dg:bdd1}
is a standard bound for the interior penalty terms; see
\cite{jang2020spatially,jang2023priori}.

\begin{theorem}[Stability of the semi-discrete problem]
Assume that $\alpha_0>0$ is sufficiently large and $\beta_0(d-1)\ge 1$ so that the SIPG bilinear form is coercive. Suppose that
\begin{align*}
\boldsymbol u_h \in H^2(0,T;L_2(\Omega))\cap W^1_\infty(0,T;\Dgvh),\qq
\boldsymbol \zeta_{hq} \in W^1_\infty(0,T;\Dgvh),
\qquad q=1,\ldots,N_\varphi,
\end{align*}
and that $(\boldsymbol u_h,\{\boldsymbol\zeta_{hq}\}_{q=1}^{N_\varphi})$ satisfies
\eqref{eq:weak1}-\eqref{eq:weak2}. Then there exists a positive constant $C$, independent of $h$, such that, for all $t\in[0,T]$,
\begin{align}
\lVert \rho^{1/2}&\dot{\boldsymbol u}_h(t)\rVert_{L_2(\Omega)}^2
+
\enorm{\boldsymbol u_h(t)}^2
+
\sum_{q=1}^{N_\varphi}
\enorm{\boldsymbol\zeta_{hq}(t)}^2
+
\sum_{q=1}^{N_\varphi}
\int_0^t
\enorm{\boldsymbol\zeta_{hq}(s)}^2\,ds
+
\int_0^t
\jump{\dot{\boldsymbol u}_h(s)}{\dot{\boldsymbol u}_h(s)}\,ds
\nonumber\\
\le&
CT
\bigg(
\lVert \rho^{1/2}\boldsymbol w_0\rVert_{L_2(\Omega)}^2
+
\enorm{\boldsymbol u_0}^2
+
\lVert \boldsymbol f\rVert_{L_2(0,T;L_2(\Omega))}^2
+
h^{-1}\lVert \boldsymbol g_N\rVert_{L_\infty(0,T;L_2(\Gamma_N))}^2\nonumber\\
&+
h^{-1}\lVert \dot{\boldsymbol g}_N\rVert_{L_2(0,T;L_2(\Gamma_N))}^2
\bigg).
\label{eq:sd:stability}
\end{align}
Here $C$ depends on $\Omega$, the polynomial degree, and the material parameters, but is independent of $h$ and of the semi-discrete solution.
\end{theorem}

\begin{proof}
Taking $\boldsymbol v=\dot{\boldsymbol u}_h$ in \eqref{eq:weak1} gives
\begin{align}
(\rho\ddot{\boldsymbol u}_h,\dot{\boldsymbol u}_h)
+\varphi_0\dginner{\boldsymbol u_h}{\dot{\boldsymbol u}_h}
+\sum_{q=1}^{N_\varphi}\dginner{\boldsymbol\zeta_{hq}}{\dot{\boldsymbol u}_h}
+\jump{\dot{\boldsymbol u}_h}{\dot{\boldsymbol u}_h}
=
F(t;\dot{\boldsymbol u}_h).
\label{eq:sd:stab:proof1}
\end{align}
Next, for each $q$, taking the DG inner product of \eqref{eq:weak2} with
\[
\boldsymbol v=\frac{1}{\tau_q\varphi_q}\boldsymbol\zeta_{hq}
\]
yields
\begin{equation}
\dginner{\boldsymbol\zeta_{hq}}{\dot{\boldsymbol u}_h}
=\frac{1}{2\varphi_q}
\frac{d}{dt}
\dginner{\boldsymbol\zeta_{hq}}{\boldsymbol\zeta_{hq}}
+\frac{1}{\tau_q\varphi_q}
\dginner{\boldsymbol\zeta_{hq}}{\boldsymbol\zeta_{hq}} .
\label{eq:sd:stab:zeta}
\end{equation}
Substituting \eqref{eq:sd:stab:zeta} into \eqref{eq:sd:stab:proof1} and using the symmetry of the SIPG bilinear form, we obtain
\begin{align}
\frac{d}{dt}&
\bigg[
\frac{\rho}{2}\lVert\dot{\boldsymbol u}_h\rVert_{L_2(\Omega)}^2
+\frac{\varphi_0}{2}\dginner{\boldsymbol u_h}{\boldsymbol u_h}
+\frac{1}{2}
\sum_{q=1}^{N_\varphi}
\frac{1}{\varphi_q}
\dginner{\boldsymbol\zeta_{hq}}{\boldsymbol\zeta_{hq}}
\bigg]
+\sum_{q=1}^{N_\varphi}
\frac{1}{\tau_q\varphi_q}
\dginner{\boldsymbol\zeta_{hq}}{\boldsymbol\zeta_{hq}}
+\jump{\dot{\boldsymbol u}_h}{\dot{\boldsymbol u}_h}\nonumber\\
=&F(t;\dot{\boldsymbol u}_h).
\label{eq:sd:energy:identity}
\end{align}
The remaining proof follows by integrating \eqref{eq:sd:energy:identity} over $(0,t)$, applying the coercivity and continuity of the DG bilinear form, and estimating the linear functional $F$ by the Cauchy-Schwarz, trace, and Young inequalities. 

Integrating \eqref{eq:sd:energy:identity} over $(0,t)$ and using the initial condition
$\boldsymbol\zeta_{hq}(0)=\boldsymbol0$ give
\begin{align}
\frac{\rho}{2}&
\lVert\dot{\boldsymbol u}_h(t)\rVert_{L_2(\Omega)}^2
+\frac{\varphi_0}{2}
\dginner{\boldsymbol u_h(t)}{\boldsymbol u_h(t)}
+\frac{1}{2}
\sum_{q=1}^{N_\varphi}
\frac{1}{\varphi_q}
\dginner{\boldsymbol\zeta_{hq}(t)}{\boldsymbol\zeta_{hq}(t)}
\nonumber\\
&+\sum_{q=1}^{N_\varphi}
\frac{1}{\tau_q\varphi_q}
\int_0^t\dginner{\boldsymbol\zeta_{hq}(s)}{\boldsymbol\zeta_{hq}(s)}\,ds
+\int_0^t\jump{\dot{\boldsymbol u}_h(s)}{\dot{\boldsymbol u}_h(s)}\,ds
\nonumber\\
=&
\frac{\rho}{2}
\lVert\boldsymbol w_0\rVert_{L_2(\Omega)}^2
+\frac{\varphi_0}{2}
\dginner{\boldsymbol u_0}{\boldsymbol u_0}
+\int_0^tF(s;\dot{\boldsymbol u}_h(s))\,ds .
\label{eq:sd:energy:integrated}
\end{align}
Using \eqref{eq:relation:dgnorm}, we rewrite the DG bilinear terms in
\eqref{eq:sd:energy:integrated} in terms of the DG energy norm. This gives
\begin{align}
\frac{\rho}{2}&
\lVert\dot{\boldsymbol u}_h(t)\rVert_{L_2(\Omega)}^2
+\frac{\varphi_0}{2}
\enorm{\boldsymbol u_h(t)}^2
+\frac{1}{2}
\sum_{q=1}^{N_\varphi}
\frac{1}{\varphi_q}
\enorm{\boldsymbol\zeta_{hq}(t)}^2
\nonumber\\
&+\sum_{q=1}^{N_\varphi}
\frac{1}{\tau_q\varphi_q}
\int_0^t\enorm{\boldsymbol\zeta_{hq}(s)}^2\,ds
+\int_0^t\jump{\dot{\boldsymbol u}_h(s)}{\dot{\boldsymbol u}_h(s)}\,ds
\nonumber\\
=&\frac{\rho}{2}
\lVert\boldsymbol w_0\rVert_{L_2(\Omega)}^2
+\frac{\varphi_0}{2}
\dginner{\boldsymbol u_0}{\boldsymbol u_0}
+\int_0^tF(s;\dot{\boldsymbol u}_h(s))\,ds+\varphi_0\sum_{e\subset\Gamma_h\cup\Gamma_D}\int_e\{\ushort{\boldsymbol D}\ushort{\boldsymbol\varepsilon}(\boldsymbol u_h(t))\}:\njump{\boldsymbol u_h(t)}\,de
\nonumber\\
&+\sum_{q=1}^{N_\varphi}\frac{1}{\varphi_q}\sum_{e\subset\Gamma_h\cup\Gamma_D}\int_e\{\ushort{\boldsymbol D}\ushort{\boldsymbol\varepsilon}(\boldsymbol\zeta_{hq}(t))\}:\njump{\boldsymbol\zeta_{hq}(t)}\,de
\nonumber\\
&+2\sum_{q=1}^{N_\varphi}\frac{1}{\tau_q\varphi_q}\int_0^t\sum_{e\subset\Gamma_h\cup\Gamma_D}\int_e\{\ushort{\boldsymbol D}\ushort{\boldsymbol\varepsilon}(\boldsymbol\zeta_{hq}(s))\}:\njump{\boldsymbol\zeta_{hq}(s)}\,de\,ds .
\label{eq:sd:energy:enorm}
\end{align}
By \eqref{eq:vector:dg:bdd2}, the last three terms on the right-hand side of
\eqref{eq:sd:energy:enorm} are bounded by
\begin{align}
&
\frac{C\varphi_0}{\sqrt{\alpha_0}}
\enorm{\boldsymbol u_h(t)}^2
+
\sum_{q=1}^{N_\varphi}
\frac{C}{\varphi_q\sqrt{\alpha_0}}
\enorm{\boldsymbol\zeta_{hq}(t)}^2
+
\sum_{q=1}^{N_\varphi}
\frac{C}{\tau_q\varphi_q\sqrt{\alpha_0}}
\int_0^t
\enorm{\boldsymbol\zeta_{hq}(s)}^2\,ds .
\label{eq:sd:penalty:bound}
\end{align}
Combining \eqref{eq:sd:energy:enorm} and \eqref{eq:sd:penalty:bound}, and using the continuity of the DG bilinear form for the initial term, we obtain
\begin{align}
&
\frac{\rho}{2}
\llnorm{\dot{\boldsymbol u}_h(t)}^2
+
\left(
\frac{\varphi_0}{2}
-
\frac{C\varphi_0}{\sqrt{\alpha_0}}
\right)
\enorm{\boldsymbol u_h(t)}^2
+
\sum_{q=1}^{\Nphi}
\left(
\frac{1}{2\varphi_q}
-
\frac{C}{\varphi_q\sqrt{\alpha_0}}
\right)
\enorm{\boldsymbol\zeta_{hq}(t)}^2
\nonumber\\
&
+
\sum_{q=1}^{\Nphi}
\left(
\frac{1}{\tau_q\varphi_q}
-
\frac{C}{\tau_q\varphi_q\sqrt{\alpha_0}}
\right)
\int_0^t
\enorm{\boldsymbol\zeta_{hq}(s)}^2\,ds
+
\int_0^t
\jump{\dot{\boldsymbol u}_h(s)}{\dot{\boldsymbol u}_h(s)}\,ds
\nonumber\\
&\le
\frac{\rho}{2}
\llnorm{\boldsymbol w_0}^2
+
C\enorm{\boldsymbol u_0}^2
+
\int_0^t
F(s;\dot{\boldsymbol u}_h(s))\,ds .
\label{eq:sd:stab:preF}
\end{align}
It remains to estimate the forcing term. Since
\[
F(t;\boldsymbol v)
=
\inner{\boldsymbol f(t)}{\boldsymbol v}
+
\ginner{\boldsymbol g_N(t)}{\boldsymbol v},
\]
we have
\begin{align}
\int_0^t F(s;\dot{\boldsymbol u}_h(s))\,ds
&=
\int_0^t
\inner{\boldsymbol f(s)}{\dot{\boldsymbol u}_h(s)}\,ds
+
\int_0^t
\ginner{\boldsymbol g_N(s)}{\dot{\boldsymbol u}_h(s)}\,ds .
\label{eq:sd:Fsplit}
\end{align}
For the body-force term, the Cauchy-Schwarz and Young inequalities give, for any $\epsilon_a>0$,
\begin{align}
\int_0^t
\inner{\boldsymbol f(s)}{\dot{\boldsymbol u}_h(s)}\,ds
\le
\frac{1}{2\epsilon_a}
\Llnorm{\boldsymbol f}^2
+
\frac{\epsilon_a}{2}
\int_0^t
\llnorm{\dot{\boldsymbol u}_h(s)}^2\,ds .
\label{eq:sd:fbound}
\end{align}
For the boundary term, integration by parts in time gives
\begin{align}
\int_0^t
\ginner{\boldsymbol g_N(s)}{\dot{\boldsymbol u}_h(s)}\,ds
&=
\ginner{\boldsymbol g_N(t)}{\boldsymbol u_h(t)}
-
\ginner{\boldsymbol g_N(0)}{\boldsymbol u_0}
-
\int_0^t
\ginner{\dot{\boldsymbol g}_N(s)}{\boldsymbol u_h(s)}\,ds .
\label{eq:sd:gibp}
\end{align}
Using the inverse trace inequality
\[
\lVert\boldsymbol v\rVert_{L_2(\Gamma_N)}^2
\le
Ch^{-1}\enorm{\boldsymbol v}^2,
\qquad
\boldsymbol v\in\Dgvh,
\]
and Young's inequality, we obtain, for any $\epsilon_b>0$,
\begin{align}
\left|
\ginner{\boldsymbol g_N(t)}{\boldsymbol u_h(t)}
\right|
&\le
\frac{1}{2\epsilon_b}
\gnorm{\boldsymbol g_N(t)}^2
+
\frac{C\epsilon_b}{2h}
\enorm{\boldsymbol u_h(t)}^2,
\label{eq:sd:gtbound}
\\
\left|
\ginner{\boldsymbol g_N(0)}{\boldsymbol u_0}
\right|
&\le
\frac{C}{h}
\gnorm{\boldsymbol g_N(0)}^2
+
C\enorm{\boldsymbol u_0}^2,
\label{eq:sd:g0bound}
\\
\left|
\int_0^t
\ginner{\dot{\boldsymbol g}_N(s)}{\boldsymbol u_h(s)}\,ds
\right|
&\le
\frac{1}{2\epsilon_b}
\lgnorm{\dot{\boldsymbol g}_N}^2
+
\frac{C\epsilon_b}{2h}
\int_0^t
\enorm{\boldsymbol u_h(s)}^2\,ds .
\label{eq:sd:gdotbound}
\end{align}
Combining \eqref{eq:sd:Fsplit}-\eqref{eq:sd:gdotbound}, we obtain
\begin{align}
\int_0^t
F(s;\dot{\boldsymbol u}_h(s))\,ds
\leq&
\frac{1}{2\epsilon_a}
\Llnorm{\boldsymbol f}^2
+
\frac{\epsilon_a}{2}
\int_0^t
\llnorm{\dot{\boldsymbol u}_h(s)}^2\,ds
\nonumber\\
&+
\left(
\frac{1}{2\epsilon_b}
+
\frac{C}{h}
\right)
\ignorm{\boldsymbol g_N}^2
+
\frac{1}{2\epsilon_b}
\lgnorm{\dot{\boldsymbol g}_N}^2
\nonumber\\
&+
C\enorm{\boldsymbol u_0}^2
+
\frac{C\epsilon_b}{2h}
\enorm{\boldsymbol u_h(t)}^2
+
\frac{C\epsilon_b}{2h}
\int_0^t
\enorm{\boldsymbol u_h(s)}^2\,ds .
\label{eq:sd:Fbound}
\end{align}
Substituting \eqref{eq:sd:Fbound} into \eqref{eq:sd:stab:preF} yields
\begin{align}
\frac{\rho}{2}&
\llnorm{\dot{\boldsymbol u}_h(t)}^2
+\left(
\frac{\varphi_0}{2}
-\frac{C\varphi_0}{\sqrt{\alpha_0}}
-\frac{C\epsilon_b}{2h}
\right)
\enorm{\boldsymbol u_h(t)}^2+\sum_{q=1}^{\Nphi}
\left(
\frac{1}{2\varphi_q}
-\frac{C}{\varphi_q\sqrt{\alpha_0}}
\right)
\enorm{\boldsymbol\zeta_{hq}(t)}^2
\nonumber\\
&+\sum_{q=1}^{\Nphi}
\left(
\frac{1}{\tau_q\varphi_q}
-\frac{C}{\tau_q\varphi_q\sqrt{\alpha_0}}
\right)
\int_0^t
\enorm{\boldsymbol\zeta_{hq}(s)}^2\,ds
+\int_0^t
\jump{\dot{\boldsymbol u}_h(s)}{\dot{\boldsymbol u}_h(s)}\,ds
\nonumber\\
\leq&C
\bigg(
\llnorm{\boldsymbol w_0}^2
+\enorm{\boldsymbol u_0}^2
+\frac{1}{\epsilon_a}
\Llnorm{\boldsymbol f}^2
+\left(
\frac{1}{\epsilon_b}
+\frac{1}{h}
\right)
\ignorm{\boldsymbol g_N}^2
+\frac{1}{\epsilon_b}
\lgnorm{\dot{\boldsymbol g}_N}^2
\bigg)
\nonumber\\
&+\frac{\epsilon_a}{2}
\int_0^t
\llnorm{\dot{\boldsymbol u}_h(s)}^2\,ds
+\frac{C\epsilon_b}{2h}
\int_0^t\enorm{\boldsymbol u_h(s)}^2\,ds .
\label{eq:sd:stab:almost}
\end{align}
Taking the supremum over $0\leq t\leq T$ in \eqref{eq:sd:stab:almost}, and using
\[
\int_0^t \llnorm{\dot{\boldsymbol u}_h(s)}^2\,ds
\leq
T\ilnorm{\dot{\boldsymbol u}_h}^2,
\qquad
\int_0^t \enorm{\boldsymbol u_h(s)}^2\,ds
\leq
T\ienorm{\boldsymbol u_h}^2,
\]
we obtain
\begin{align}
\frac{\rho}{2}&
\ilnorm{\dot{\boldsymbol u}_h}^2
+\left(
\frac{\varphi_0}{2}
-\frac{C\varphi_0}{\sqrt{\alpha_0}}
-\frac{C\epsilon_b}{2h}
\right)
\ienorm{\boldsymbol u_h}^2
+\sum_{q=1}^{\Nphi}
\left(
\frac{1}{2\varphi_q}
-\frac{C}{\varphi_q\sqrt{\alpha_0}}
\right)
\ienorm{\boldsymbol\zeta_{hq}}^2
\nonumber\\
&+\sum_{q=1}^{\Nphi}
\left(
\frac{1}{\tau_q\varphi_q}
-\frac{C}{\tau_q\varphi_q\sqrt{\alpha_0}}
\right)
\lenorm{\boldsymbol\zeta_{hq}}^2
+\int_0^T
\jump{\dot{\boldsymbol u}_h(s)}{\dot{\boldsymbol u}_h(s)}\,ds
\nonumber\\
\leq&C\bigg(
\llnorm{\boldsymbol w_0}^2
+\enorm{\boldsymbol u_0}^2
+\frac{1}{\epsilon_a}
\Llnorm{\boldsymbol f}^2
+\left(
\frac{1}{\epsilon_b}
+\frac{1}{h}
\right)
\ignorm{\boldsymbol g_N}^2
+\frac{1}{\epsilon_b}
\lgnorm{\dot{\boldsymbol g}_N}^2
\bigg)
\nonumber\\
&+\frac{\epsilon_a T}{2}
\ilnorm{\dot{\boldsymbol u}_h}^2
+\frac{C\epsilon_b T}{2h}
\ienorm{\boldsymbol u_h}^2 .
\label{eq:sd:stab:sup}
\end{align}
We now choose
\[
\epsilon_a=\frac{\rho}{4T},
\qquad
\epsilon_b=\frac{\varphi_0 h}{4C(T+1)} .
\]
Then the last two terms on the right-hand side of \eqref{eq:sd:stab:sup} can be absorbed into the corresponding terms on the left-hand side. Moreover, by taking $\alpha_0$ sufficiently large, all coefficients remain positive. Therefore, there exists a constant $C>0$, independent of $h$ and the semi-discrete solution, such that
\begin{align*}
&
\ilnorm{\rho^{1/2}\dot{\boldsymbol u}_h}^2
+
\ienorm{\boldsymbol u_h}^2
+
\sum_{q=1}^{\Nphi}
\ienorm{\boldsymbol\zeta_{hq}}^2
+
\sum_{q=1}^{\Nphi}
\lenorm{\boldsymbol\zeta_{hq}}^2
\\&+
\int_0^T
\jump{\dot{\boldsymbol u}_h(s)}{\dot{\boldsymbol u}_h(s)}\,ds
\nonumber\\
\leq&
C
\bigg(
\llnorm{\rho^{1/2}\boldsymbol w_0}^2
+
\enorm{\boldsymbol u_0}^2
+
T\Llnorm{\boldsymbol f}^2
+
h^{-1}(T+1)\ignorm{\boldsymbol g_N}^2
\\&+
h^{-1}(T+1)\lgnorm{\dot{\boldsymbol g}_N}^2
\bigg).
\end{align*}

\end{proof}

The stability estimate immediately yields the well-posedness of the semi-discrete problem. Furthermore, since the analysis avoids the use of Gr\"onwall's inequality, the resulting stability bound grows at most linearly with the final time $T$, rather than exponentially. Consequently, the estimate remains valid over long time intervals and is therefore well suited for long-time simulations.

We next derive \textit{a priori} error estimates for the semi-discrete approximation. Following the standard DG framework, we introduce the DG elliptic projection and decompose the error into projection and discrete components.

The DG elliptic projection $\dgelliptic:\Dgv\to\Dgvh$ is defined by
\[
\dginner{\dgelliptic\boldsymbol u}{\boldsymbol v}
=
\dginner{\boldsymbol u}{\boldsymbol v},
\qquad
\forall \boldsymbol v\in\Dgvh .
\]
For $\boldsymbol u\in\Dgv$ with $s>3/2$, the standard DG approximation estimates give
\begin{align}
\enorm{\boldsymbol u-\dgelliptic\boldsymbol u}
&\leq
C h^{\min(k+1,s)-1}
\vertiii{\boldsymbol u}_{H^s(\mathcal E_h)},
\label{vector:appDG}
\\
\llnorm{\boldsymbol u-\dgelliptic\boldsymbol u}
&\leq
C h^{\min(k+1,s)}
\vertiii{\boldsymbol u}_{H^s(\mathcal E_h)} .
\label{vector:appl2}
\end{align}
Here $C>0$ is independent of $h$ and $\boldsymbol u$, provided the penalty parameters are sufficiently large.  The $L_2$ estimate \eqref{vector:appl2} is understood under the standard elliptic regularity assumption for the associated dual problem; see, e.g., \cite{riviere2003discontinuous,houston2006hp}.

We decompose the errors into projection and discrete components. For the displacement, we define
\[
\boldsymbol\theta(t)
:=
\boldsymbol u(t)-\dgelliptic\boldsymbol u(t),
\qquad
\boldsymbol\chi(t)
:=
\dgelliptic\boldsymbol u(t)-\boldsymbol u_h(t).
\]
For the internal variables, we define, for each $q=1,\ldots,\Nphi$,
\[
\boldsymbol\nu_q(t)
:=
\boldsymbol\zeta_q(t)-\dgelliptic\boldsymbol\zeta_q(t),
\qquad
\boldsymbol\Upsilon_q(t)
:=
\dgelliptic\boldsymbol\zeta_q(t)-\boldsymbol\zeta_{hq}(t).
\]
Then
\[
\boldsymbol u(t)-\boldsymbol u_h(t)
=
\boldsymbol\theta(t)+\boldsymbol\chi(t),
\qquad
\boldsymbol\zeta_q(t)-\boldsymbol\zeta_{hq}(t)
=
\boldsymbol\nu_q(t)+\boldsymbol\Upsilon_q(t).
\]
By the definition of the DG elliptic projection, the projection errors satisfy
\[
\dginner{\boldsymbol\theta(t)}{\boldsymbol v}=0,
\qquad
\dginner{\boldsymbol\nu_q(t)}{\boldsymbol v}=0,
\qquad
\forall \boldsymbol v\in\Dgvh.
\]
Moreover, the projection estimates \eqref{vector:appDG}-\eqref{vector:appl2} imply
\[
\enorm{\boldsymbol\theta(t)}
+
\sum_{q=1}^{\Nphi}\enorm{\boldsymbol\nu_q(t)}
\leq
C h^{\min(k+1,s)-1}
\left(
\vertiii{\boldsymbol u(t)}_{H^s(\mathcal E_h)}
+
\sum_{q=1}^{\Nphi}
\vertiii{\boldsymbol\zeta_q(t)}_{H^s(\mathcal E_h)}
\right),
\]
and
\[
\llnorm{\boldsymbol\theta(t)}
+
\sum_{q=1}^{\Nphi}\llnorm{\boldsymbol\nu_q(t)}
\leq
C h^{\min(k+1,s)}
\left(
\vertiii{\boldsymbol u(t)}_{H^s(\mathcal E_h)}
+
\sum_{q=1}^{\Nphi}
\vertiii{\boldsymbol\zeta_q(t)}_{H^s(\mathcal E_h)}
\right).
\]
Thus, the remaining task is to estimate the discrete errors
$\boldsymbol\chi$ and $\boldsymbol\Upsilon_q$.

Subtracting the semi-discrete formulation from the continuous variational problem and using the projection orthogonality, we obtain, for all $\boldsymbol v\in\Dgvh$,
\begin{gather}
(\rho\ddot{\boldsymbol\chi},\boldsymbol v)
+
\varphi_0\dginner{\boldsymbol\chi}{\boldsymbol v}
+
\sum_{q=1}^{\Nphi}
\dginner{\boldsymbol\Upsilon_q}{\boldsymbol v}
+
\jump{\dot{\boldsymbol\chi}}{\boldsymbol v}
=
-
(\rho\ddot{\boldsymbol\theta},\boldsymbol v)
-
\jump{\dot{\boldsymbol\theta}}{\boldsymbol v},
\label{eq:error:chi}\\
\tau_q\dginner{\dot{\boldsymbol\Upsilon}_q}{\boldsymbol v}
+
\dginner{\boldsymbol\Upsilon_q}{\boldsymbol v}
=
\tau_q\varphi_q\dginner{\dot{\boldsymbol\chi}}{\boldsymbol v}
+
\tau_q\varphi_q\dginner{\dot{\boldsymbol\theta}}{\boldsymbol v}
-
\tau_q\dginner{\dot{\boldsymbol\nu}_q}{\boldsymbol v},
\qquad q=1,\ldots,\Nphi.
\label{eq:error:upsilon}
\end{gather}

Taking $\boldsymbol v=\dot{\boldsymbol\chi}$ in \eqref{eq:error:chi} and
$\boldsymbol v=\boldsymbol\Upsilon_q/(\tau_q\varphi_q)$ in
\eqref{eq:error:upsilon}, summing over $q=1,\ldots,\Nphi$, and using the symmetry
of the SIPG bilinear form, we obtain
\begin{align}
\frac{d}{dt}
\bigg[
\frac{\rho}{2}\llnorm{\dot{\boldsymbol\chi}}^2
+
\frac{\varphi_0}{2}\dginner{\boldsymbol\chi}{\boldsymbol\chi}
+
\frac12\sum_{q=1}^{\Nphi}
\frac{1}{\varphi_q}
\dginner{\boldsymbol\Upsilon_q}{\boldsymbol\Upsilon_q}
\bigg]
+
\sum_{q=1}^{\Nphi}
\frac{1}{\tau_q\varphi_q}
\dginner{\boldsymbol\Upsilon_q}{\boldsymbol\Upsilon_q}
+
\jump{\dot{\boldsymbol\chi}}{\dot{\boldsymbol\chi}}
\nonumber\\
=
-\inner{\rho\ddot{\boldsymbol\theta}}{\dot{\boldsymbol\chi}}
-\jump{\dot{\boldsymbol\theta}}{\dot{\boldsymbol\chi}}
+
\sum_{q=1}^{\Nphi}
\dginner{\boldsymbol\Upsilon_q}{\dot{\boldsymbol\theta}}
-
\sum_{q=1}^{\Nphi}
\frac{1}{\varphi_q}
\dginner{\boldsymbol\Upsilon_q}{\dot{\boldsymbol\nu}_q}.
\label{eq:error:energy}
\end{align}
Using the same argument as in the proof of the stability estimate, together with
Cauchy-Schwarz and Young's inequalities, we obtain
\begin{align}
&
\ilnorm{\rho^{1/2}\dot{\boldsymbol\chi}}^2
+
\ienorm{\boldsymbol\chi}^2
+
\sum_{q=1}^{\Nphi}\ienorm{\boldsymbol\Upsilon_q}^2
+
\sum_{q=1}^{\Nphi}\lenorm{\boldsymbol\Upsilon_q}^2
+
\int_0^T
\jump{\dot{\boldsymbol\chi}(s)}{\dot{\boldsymbol\chi}(s)}\,ds
\nonumber\\
&\leq
C
\bigg(
\Llnorm{\ddot{\boldsymbol\theta}}^2
+
\int_0^T
\jump{\dot{\boldsymbol\theta}(s)}{\dot{\boldsymbol\theta}(s)}\,ds
+
\sum_{q=1}^{\Nphi}
\lenorm{\dot{\boldsymbol\theta}}^2
+
\sum_{q=1}^{\Nphi}
\lenorm{\dot{\boldsymbol\nu}_q}^2
\bigg).
\label{eq:error:discrete:bound}
\end{align}
The projection estimates \eqref{vector:appDG}-\eqref{vector:appl2}, applied also to
time derivatives, imply
\begin{align}
&\Llnorm{\ddot{\boldsymbol\theta}}
+
\lenorm{\dot{\boldsymbol\theta}}
+
\sum_{q=1}^{\Nphi}
\lenorm{\dot{\boldsymbol\nu}_q}\nonumber\\
&\leq
C h^{\min(k+1,s)-1}
\bigg(
\lVert \ddot{\boldsymbol u}\rVert_{L_2(0,T;\boldsymbol H^s(\mathcal E_h))}
+
\lVert \dot{\boldsymbol u}\rVert_{L_2(0,T;\boldsymbol H^s(\mathcal E_h))}
+
\sum_{q=1}^{\Nphi}
\lVert \dot{\boldsymbol\zeta}_q\rVert_{L_2(0,T;\boldsymbol H^s(\mathcal E_h))}
\bigg).
\label{eq:projection:time:bound}
\end{align}
Moreover,
\[
\int_0^T
\jump{\dot{\boldsymbol\theta}(s)}{\dot{\boldsymbol\theta}(s)}\,ds
\leq
C h^{2(\min(k+1,s)-1)}
\lVert \dot{\boldsymbol u}\rVert_{L_2(0,T;\boldsymbol H^s(\mathcal E_h))}^2 .
\]
Combining these estimates with the triangle inequality gives the following result.

\begin{theorem}\label{thm:semi:error}
Assume that $\alpha_0>0$ is sufficiently large and $\beta_0(d-1)\ge1$. Let
$\boldsymbol u$ and $\boldsymbol\zeta_q$, $q=1,\ldots,\Nphi$, be sufficiently smooth, and let
$\boldsymbol u_h$ and $\boldsymbol\zeta_{hq}$ be the semi-discrete DG approximations.
Then there exists a constant $C>0$, independent of $h$, such that
\begin{align}
&
\ilnorm{\rho^{1/2}(\dot{\boldsymbol u}-\dot{\boldsymbol u}_h)}
+
\ienorm{\boldsymbol u-\boldsymbol u_h}
+
\sum_{q=1}^{\Nphi}
\ienorm{\boldsymbol\zeta_q-\boldsymbol\zeta_{hq}}
+
\sum_{q=1}^{\Nphi}
\lenorm{\boldsymbol\zeta_q-\boldsymbol\zeta_{hq}}
\nonumber\\
&\leq
C h^{\min(k+1,s)-1}
\bigg(
\lVert \boldsymbol u\rVert_{W^1_\infty(0,T;\boldsymbol H^s(\mathcal E_h))}
+
\lVert \ddot{\boldsymbol u}\rVert_{L_2(0,T;\boldsymbol H^s(\mathcal E_h))}
+
\sum_{q=1}^{\Nphi}
\lVert \boldsymbol\zeta_q\rVert_{W^1_\infty(0,T;\boldsymbol H^s(\mathcal E_h))}
\bigg).
\label{eq:apriori:final}
\end{align}
\end{theorem}

\begin{rmk}
Theorem~\ref{thm:semi:error} gives the optimal convergence rate in the DG energy norm. Under the standard elliptic regularity assumption for the associated dual problem, an optimal $L_2$-error estimate for the displacement can also be obtained by a classical Aubin-Nitsche argument.
\end{rmk}

\subsection{Time discretization}

We now discretize the semi-discrete system \eqref{eq:firstorder:compact} in time. Let
\[
0=t_0<t_1<\cdots<t_N=T,
\qquad
\Delta t=t_{n+1}-t_n,
\]
be a uniform partition of $[0,T]$, and let $\boldsymbol y^n$ denote an approximation to
$\boldsymbol y(t_n)$.

In the numerical experiments, we consider the backward Euler and Crank-Nicolson schemes. The backward Euler method applied to \eqref{eq:firstorder:compact} reads
\begin{equation}
\left(I-\Delta t\,\mathcal A\right)\boldsymbol y^{n+1}
=
\boldsymbol y^n
+
\Delta t\,\boldsymbol g(t_{n+1}).
\label{eq:BE:matrix}
\end{equation}
The Crank-Nicolson method is given by
\begin{equation}
\left(I-\frac{\Delta t}{2}\mathcal A\right)\boldsymbol y^{n+1}
=
\left(I+\frac{\Delta t}{2}\mathcal A\right)\boldsymbol y^n
+
\frac{\Delta t}{2}
\left(
\boldsymbol g(t_{n+1})
+
\boldsymbol g(t_n)
\right).
\label{eq:CN:matrix}
\end{equation}

For sufficiently smooth solutions, the backward Euler method is first-order accurate in time, whereas the Crank-Nicolson method is second-order accurate. Hence, the temporal discrete error is expected to behave as
\[
O\left(\Delta t^r\right),
\qquad
r=
\begin{cases}
1, & \text{backward Euler},\\
2, & \text{Crank-Nicolson},
\end{cases}
\]
under the regularity assumptions
$\boldsymbol y\in C^2([0,T])$ for backward Euler and
$\boldsymbol y\in C^3([0,T])$ for Crank-Nicolson, respectively.
Consequently, combining the temporal and spatial discretization errors yields the estimate
\[
\|\boldsymbol y(t_n)-\boldsymbol y_h^n\|
=
O\left(
h^{\min(k+1,s)}
+
\Delta t^r
\right),\qq\text{and}\qq\|\boldsymbol y(t_n)-\boldsymbol y_h^n\|_{\mathcal A}
=
O\left(
h^{\min(k+1,s)-1}
+
\Delta t^r
\right).
\]

\section{Proper orthogonal decomposition}\label{sec:pod}
In this section, we introduce the proposed proper orthogonal decomposition (POD) framework for reducing the computational complexity of the DG approximation. We first review the POD formulation and then derive the corresponding DG-POD reduced-order model.
\subsection{Reduced-order POD system}
\label{ssec:rom}

We now construct a reduced-order model (ROM) for the semi-discrete system \eqref{eq:firstorder:compact}, e.g. \cite{Si1,kunisch,SS2}. Let
\[
0\leq T_0<T_1<\cdots<T_{N_s}=T
\]
be a collection of snapshot times and define the snapshot matrix
\[
Y=
\begin{bmatrix}
\boldsymbol y(T_0) &
\boldsymbol y(T_1) &
\cdots &
\boldsymbol y(T_{N_s})
\end{bmatrix}
\in\mathbb R^{N\times(N_s+1)},
\]
where $N=(\Nphi+2)\ndof$ denotes the dimension of the full-order system.

For a prescribed reduced dimension $\ell$, the POD basis is defined as the solution of
\begin{equation}
\min_{\{\boldsymbol\psi_j\}_{j=1}^{\ell}}
\sum_{n=0}^{N_s}
\left\|
\boldsymbol y(T_n)
-
\sum_{j=1}^{\ell}
\inner{\boldsymbol y(T_n)}{\boldsymbol\psi_j}
\boldsymbol\psi_j
\right\|^2,\qq\text{subject to $\inner{\boldsymbol\psi_i}{\boldsymbol\psi_j}
=
\delta_{ij},
\qquad
i,j=1,\ldots,\ell$.}
\label{eq:pod:min}
\end{equation}

Let
\[
Y=U\Sigma V^\top,
\qquad
\Sigma=\operatorname{diag}(\sigma_1,\ldots,\sigma_m),
\qquad
\sigma_1\ge\cdots\ge\sigma_m>0,
\]
be the singular value decomposition of $Y$, where
$m=\operatorname{rank}(Y)$.
The optimal POD basis is then given by the first $\ell$ left singular vectors,
\[
\Phi_\ell=
\begin{bmatrix}
\boldsymbol\psi_1 & \cdots & \boldsymbol\psi_\ell
\end{bmatrix}
=
U^\ell.
\]
Equivalently, the POD modes may be obtained from the correlation matrix
\[
K=Y^\top Y=V\Sigma^2V^\top.
\]
Since $K\in\mathbb R^{(N_s+1)\times(N_s+1)}$, solving the associated eigenvalue problem is considerably cheaper than diagonalizing $YY^\top\in\mathbb R^{N\times N}$ whenever $N_s\ll N$. Let
$V_\ell=[v_1,\ldots,v_\ell]$
denote the eigenvectors corresponding to the $\ell$ largest eigenvalues
$\sigma_1^2,\ldots,\sigma_\ell^2$.
The POD basis matrix is then constructed as
\[
\Phi_\ell
=
YV_\ell\Sigma_\ell^{-1},
\qq\text{where}\qq
\Sigma_\ell
=
\operatorname{diag}(\sigma_1,\ldots,\sigma_\ell).
\]
By construction, the columns of $\Phi_\ell$ are orthonormal and span the dominant snapshot subspace.

The POD basis minimizes the mean-square projection error among all $\ell$-dimensional subspaces and satisfies
\begin{equation}
\sum_{n=0}^{N_s}
\left\|
\boldsymbol y(t_n)
-
\Phi_\ell\Phi_\ell^\top\boldsymbol y(t_n)
\right\|^2
=
\sum_{j=\ell+1}^{m}
\sigma_j^2.
\label{eq:pod:error}
\end{equation}

The reduced approximation is sought in the form
\[
\boldsymbol y_\ell(t)=\Phi_\ell\boldsymbol a(t),
\qquad
\boldsymbol a(t)\in\mathbb R^\ell.
\]
Substituting this expansion into \eqref{eq:firstorder:compact} and applying a Galerkin projection onto
$\operatorname{span}\{\boldsymbol\psi_1,\ldots,\boldsymbol\psi_\ell\}$ yields
\[
\dot{\boldsymbol a}(t)
=
\mathcal A_\ell\boldsymbol a(t)
+
\boldsymbol g_\ell(t),
\qquad
\mathcal A_\ell=\Phi_\ell^\top\mathcal A\Phi_\ell,
\qquad
\boldsymbol g_\ell(t)=\Phi_\ell^\top\boldsymbol g(t).
\]

Alternatively, one may construct the POD basis with respect to a weighted inner product. Let
$W\in\mathbb R^{N\times N}$ be symmetric positive definite, where
$N=(\Nphi+2)\ndof$, and define
\[
\inner{\boldsymbol x}{\boldsymbol z}_W
=
\boldsymbol x^\top W\boldsymbol z,
\qquad
\lVert\boldsymbol x\rVert_W^2
=
\boldsymbol x^\top W\boldsymbol x .
\]
The standard inner product corresponds to the special case $W=I$. For the present first-order system, a natural energy-based choice is the block-diagonal weight
\[
W_E
=
\operatorname{diag}
\left(
\varphi_0A,\,
\rho M,\,
\frac{1}{\varphi_1}A,\ldots,
\frac{1}{\varphi_{\Nphi}}A
\right),
\]
which represents the discrete counterpart of the semi-discrete energy.

The corresponding weighted POD basis is obtained from the SVD of
$\widetilde Y=W^{1/2}Y$,
or equivalently from the weighted correlation matrix
$K_W=Y^\top W Y$.
The resulting basis satisfies $\Phi_\ell^\top W\Phi_\ell=I_\ell$, and the $W$-orthogonal projection is $P_W=\Phi_\ell\Phi_\ell^\top W$.
Consequently, we have
\[
\sum_{n=0}^{N_s}
\lVert
\boldsymbol y(T_n)-P_W\boldsymbol y(T_n)
\rVert_W^2
=
\sum_{j=\ell+1}^{m}\sigma_j^2 .
\]

\subsection{Construction of the DG-POD reduced model}\label{ssec:construction}

We now summarize the implementation procedure for the DG-POD reduced-order model. Although the first-order formulation \eqref{eq:firstorder:compact} is written in terms of the full state vector, it is not necessary to construct POD modes for all state components. In the numerical implementation, we build the POD space from displacement snapshots and use the same reduced basis for the velocity and internal variables.

Let $Y_u=
\begin{bmatrix}
\boldsymbol{\mathfrak u}(T_0)&
\boldsymbol{\mathfrak u}(T_1)&
\cdots&
\boldsymbol{\mathfrak u}(T_{N_s})
\end{bmatrix}
\in\mathbb R^{\ndof\times(N_s+1)}
$
be the displacement snapshot matrix.  Given a symmetric positive definite weighting matrix
$W_u\in\mathbb R^{\ndof\times\ndof}$, the POD basis $\Phi_u\in\mathbb R^{\ndof\times\ell}$ is
computed from the weighted snapshot matrix
\[
\widetilde Y_u=W_u^{1/2}Y_u.
\]
The standard choice corresponds to $W_u=I$, while mass- or energy-weighted POD corresponds,
respectively, to choices such as $W_u=M$ or $W_u=A$.
The resulting basis satisfies
\[
\Phi_u^\top W_u\Phi_u=I_\ell.
\]

Once the POD basis has been constructed from the displacement snapshots, the fully discrete reduced solution can be computed on a possibly finer time grid. Let $\Delta t>0$ be the time-step size for the reduced simulation.
As a natural choice, we take
\[
\Delta t\ll \Delta T,
\qquad
\Delta T:=\max_{1\le i\le N_s}(T_i-T_{i-1}),
\]
where $\{T_i\}_{i=0}^{N_s}$ denotes the snapshot times used to construct the POD basis.

For the fully discrete reduced system, it is not necessary to solve simultaneously for all reduced variables. 
Once $\boldsymbol u_\ell^{n+1}$ is computed, the reduced velocity and internal variables can be updated algebraically. 
We use a one-step implicit scheme with parameter
\[
\omega=
\begin{cases}
1, & \text{backward Euler},\\
\frac12, & \text{Crank-Nicolson}.
\end{cases}
\]
Using the POD basis matrix, the reduced mass, stiffness, jump, and effective stiffness matrices are defined by
\[
M_\ell=\Phi_u^\top M\Phi_u,\qquad
A_\ell=\Phi_u^\top A\Phi_u,\qquad\text{and}\qq
J_\ell=\Phi_u^\top J\Phi_u.
\]
Similarly, the reduced load vector is given by
\[
\boldsymbol b_\ell^n
=
\Phi_u^\top\boldsymbol b^n.
\]
Hence, the reduced fully discrete equations are written as
\begin{gather}
\frac{\boldsymbol u_\ell^{n+1}-\boldsymbol u_\ell^n}{\Delta t}
=
\omega\boldsymbol w_\ell^{n+1}+(1-\omega)\boldsymbol w_\ell^{n},
\label{eq:rom:update:u:omega}
\\
M_\ell
\frac{\boldsymbol w_\ell^{n+1}-\boldsymbol w_\ell^n}{\Delta t}
+J_\ell
\frac{\boldsymbol w_\ell^{n+1}+\boldsymbol w_\ell^n}{2}
+\varphi_0A_\ell\left(\omega\boldsymbol u_\ell^{n+1}+(1-\omega)\boldsymbol u_\ell^{n}\right)\notag\\
+\sum_{q=1}^{\Nphi}
A_\ell\left(\omega\boldsymbol z_{\ell q}^{n+1}+(1-\omega)\boldsymbol z_{\ell q}^{n}\right)
=
\omega\boldsymbol b_\ell^{n+1}+(1-\omega)\boldsymbol b_\ell^{n},
\label{eq:rom:update:v:omega}
\\
\frac{\boldsymbol z_{\ell q}^{n+1}-\boldsymbol z_{\ell q}^{n}}{\Delta t}
+
\frac{1}{\tau_q}\left(\omega\boldsymbol z_{\ell q}^{n+1}+(1-\omega)\boldsymbol z_{\ell q}^{n}\right)
=
\varphi_q\left(\omega\boldsymbol w_\ell^{n+1}+(1-\omega)\boldsymbol w_\ell^{n}\right),
\qquad q=1,\ldots,\Nphi .
\label{eq:rom:update:z:omega}
\end{gather}
Here $\omega=1$ gives the backward Euler method, while $\omega=1/2$ gives the Crank-Nicolson method. The initial reduced coefficients are obtained by projecting the initial data onto the POD space:
\[
\boldsymbol u_\ell^0=\Phi_u^\dagger\boldsymbol{\mathfrak u}_0,
\qquad
\boldsymbol w_\ell^0=\Phi_u^\dagger\boldsymbol{\mathfrak w}_0,
\qquad
\boldsymbol z_{\ell q}^0=\boldsymbol 0,
\qquad q=1,\ldots,\Nphi,
\]
where $\Phi_u^\dagger=\Phi_u^\top W_u$ if the POD basis is $W_u$-orthonormal.

For each time step, the reduced displacement coefficient
$\boldsymbol u_\ell^{n+1}$ is first computed from
\eqref{eq:rom:update:u:omega}-\eqref{eq:rom:update:z:omega}. Once
$\boldsymbol u_\ell^{n+1}$ is available, the reduced velocity
$\boldsymbol w_\ell^{n+1}$ is obtained from
\eqref{eq:rom:update:u:omega}, and the internal variables
$\boldsymbol z_{\ell q}^{n+1}$ are subsequently updated from
\eqref{eq:rom:update:z:omega}. Therefore, the reduced simulation only
requires the solution of a linear system of dimension $\ell$ at each
time step, followed by inexpensive algebraic updates for the remaining
variables.

The approximate full-order coefficient vectors are finally recovered by
$\boldsymbol{\mathfrak u}_h^n\approx\Phi_u\boldsymbol u_\ell^n$.
Eventually, the DG finite element solution can be reconstructed from the approximate full-order coefficient vectors. The reconstructed DG-POD solution at time $t_n$ is the DG finite element function
\[
\boldsymbol u_{h,\ell}^n(\bs x)
=
\sum_{i=1}^{\ndof}
\mathfrak u_i^n \boldsymbol\phi_i(\bs x),
\qquad
\boldsymbol{\mathfrak u}^n
=
\Phi_u\boldsymbol u_\ell^n\in\mathbb R^{\ndof},
\]
where $\boldsymbol u_\ell^n\in\mathbb R^\ell$ is the reduced displacement coefficient vector and $\Phi_u\in\mathbb R^{\ndof\times \ell}$ is the displacement POD basis. The complete DG-POD procedure is summarized in Algorithm~\ref{alg:dgpod}.

Finally, the accuracy of the proposed DG-POD framework is determined by three sources of error: the spatial discretization error, the temporal discretization error, and the POD truncation error. 
The a priori analysis established above shows that the DG approximation converges with order $O(h^{\min(k+1,s)-1})$ in the energy norm, while the time discretization contributes an error of order $O(\Delta t^r)$, where $r=1$ for the backward Euler method and $r=2$ for the Crank-Nicolson method. 
Moreover, the POD approximation error is characterized by the neglected singular values.

To make this decomposition explicit, let $\boldsymbol u(t_n)$ denote the exact displacement, $\boldsymbol u_h(t_n)$ the semi-discrete DG approximation, and $\boldsymbol u_h^n$ the fully discrete full-order approximation. Then,
\[
\|\boldsymbol u(t_n)-\boldsymbol u_{h,\ell}^n\|_V
\le \|\boldsymbol u(t_n)-\boldsymbol u_h(t_n)\|_V
+\|\boldsymbol u_h(t_n)-\boldsymbol u_h^n\|_V
+\|\boldsymbol u_h^n-\boldsymbol u_{h,\ell}^n\|_V.
\]
Assume that the reduced evolution is uniformly stable and that the reduced-order error is controlled by the POD projection error. By the POD optimality property,
\[
\|\boldsymbol u_h^n-\boldsymbol u_{h,\ell}^n\|_V
\lesssim \left(\sum_{j=\ell+1}^{m}\sigma_j^2\right)^{1/2}.
\]
Consequently,
\begin{equation}
\|\boldsymbol u(t_n)-\boldsymbol u_{h,\ell}^n\|_V
\lesssim h^{\min(k+1,s)-1}+\Delta t^r
+\left(\sum_{j=\ell+1}^{m}\sigma_j^2\right)^{1/2}.
\label{eq:pod:error:energy}
\end{equation}
Thus, the total error consists of the spatial discretization error, the temporal discretization error, and the POD truncation error.

\begin{algorithm}[H]
\caption{DG-POD reduced-order simulation}
\label{alg:dgpod}
\begin{algorithmic}[1]
\Require Displacement snapshots $Y_u$, reduced dimension $\ell$, POD weight $W_u$, time step $\Delta t$, and final time $T$.
\Ensure Approximate full-order coefficient vectors $\boldsymbol{\mathfrak u}^n$, $\boldsymbol{\mathfrak w}^n$, and $\boldsymbol z_q^n$.

\State Form the weighted correlation matrix $K_{W_u}= Y_u^\top W_u Y_u$.

\State Compute the spectral decomposition $K_{W_u}=V\Sigma^2V^\top$.

\State Construct the POD basis $\Phi_u=Y_uV_\ell\Sigma_\ell^{-1}$, where $V_\ell$ contains the first $\ell$ eigenvectors and $\Sigma_\ell=\operatorname{diag}(\sigma_1,\ldots,\sigma_\ell)$.

\State Assemble the reduced matrices.

\State Project the initial data to obtain
$\boldsymbol u_\ell^0$,
$\boldsymbol w_\ell^0$,
and
$\boldsymbol z_{\ell q}^0$.

\For{$n=0,\ldots,N-1$}
    \State Solve for $\boldsymbol u_\ell^{n+1}$.
    \State Update $\boldsymbol w_\ell^{n+1}$ and
    $\boldsymbol z_{\ell q}^{n+1}$.
\EndFor

\State Reconstruct
$\boldsymbol{\mathfrak u}^n$,
$\boldsymbol{\mathfrak w}^n$,
and
$\boldsymbol z_q^n$.

\State Recover the finite element solution.
\end{algorithmic}
\end{algorithm}

The computational complexity of the time-marching procedure is governed by the reduced dimension $\ell$ rather than the full-order dimension $\ndof$. This reduction is particularly important for long-time simulations. The stability estimate established above shows that the semi-discrete solution remains controlled over long time intervals, with a bound that grows at most linearly in the final time rather than exponentially. However, even when such long-time integration is numerically stable, the full-order DG system may still be prohibitively expensive because a large-scale linear system must be solved at every time step. The proposed DG-POD reduction directly addresses this computational bottleneck: each time step requires solving only a reduced system of dimension $\ell$, followed by inexpensive updates of the velocity and internal variables. Therefore, when $\ell\ll\ndof$, the method reduces the wall-clock time while preserving the stable long-time behavior of the underlying DG approximation.

\section{Numerical experiments}\label{sec:numeric}
All numerical experiments are implemented in Python using FEniCSx for the finite element discretization and assembly of the DG matrices. The POD bases are computed using NumPy/SciPy dense linear algebra routines applied to the snapshot correlation matrices. In particular, the displacement snapshots are collected in the matrix $Y_u$, and the POD basis is constructed from the dominant eigenvectors of the weighted correlation matrix $K_{W_u}=Y_u^T W_uY_u$.

The full-order DG solution is advanced in time using either the backward Euler method or the Crank-Nicolson method. At each time step, the displacement coefficient vector is first computed by solving the eliminated linear system for $\boldsymbol u_h^{n+1}$. The velocity and internal variables are then updated algebraically. The DG-POD reduced-order model uses the same time-stepping formula after Galerkin projection onto the POD space. Therefore, the full-order and reduced-order simulations differ only in the dimension of the algebraic system to be solved at each time step. For simplicity, we consider uniformly distributed snapshot times, so that
\[
T_i-T_{i-1}=\Delta T=\frac{T}{N_s},
\qquad i=1,\ldots,N_s.
\]

To assess accuracy and convergence, we compute the maximum-in-time errors in the $L_2$ and broken $H^1$ norms. Spatial and temporal convergence rates are then estimated from successive refinements of $h$ and $\Delta t$, respectively.
The computational cost is divided into:
\begin{itemize}
\item \textit{Set-up time}: DG space construction, matrix assembly, and system factorization; for DG-POD, this also includes POD basis construction and reduced-system assembly and factorization.
\item \textit{Solve time}: cumulative linear-solver time during time integration; for DG-POD, this also includes reconstruction in the full DG space.
\item \textit{Total time}: the time including snapshot-generation time.
\end{itemize}
Speed-up factors are computed relative to the corresponding DG full-order simulation.

\subsection{Smooth manufactured solution}\label{ssec:numeric:smooth}
We first examine the performance of the proposed DG-POD reduced-order model for a smooth manufactured problem. The computational domain is
$\Omega=[0,1]^2$, discretized by a uniform triangular mesh obtained by subdividing each square cell of an $N_{xy}\times N_{xy}$ Cartesian partition into two triangles, and the final time is $T=10$. Let the exact displacement be
\[
\boldsymbol u(x,y,t)
=
\begin{pmatrix}
xye^{1-t}
\\
\cos(t)\sin(xy)
\end{pmatrix}.
\]
The viscoelastic parameters are set to
\[
\varphi_0=0.5,\qq
\varphi_1=0.1,\qq
\varphi_2=0.4,\qq
\tau_1=0.5,\qq
\tau_2=1.5.
\]
We take the fourth-order tensor $\ushort{\boldsymbol{D}}$ to be the identity tensor, so that $\ushort{\boldsymbol{D}}\ushort{\boldsymbol{\varepsilon}}=\ushort{\boldsymbol{\varepsilon}}$, and the Dirichlet boundary as
\[
\Gamma_D=\{(x,y)\in\partial\Omega\ |\ x=0\textrm{ or }y=0\}.
\] 
The remaining initial and boundary data, as well as the body force, are readily computed from the exact solution.

We employ DG finite elements of polynomial degree $k$. Throughout the numerical experiments, the spatial resolution is controlled by the parameter $N_{xy}$, while the corresponding $\ndof$ varies according to the mesh size and polynomial order. Since the dimension of the scalar polynomial space $P_k(T)$ on a triangle is $\dim P_k(T)={(k+1)(k+2)}/{2}$,
the total number of degrees of freedom is given by
$\ndof=2N_{xy}^2(k+1)(k+2)$.

\paragraph{Spatial convergence}
We verify the spatial convergence of the proposed DG-POD approximation and compare it with the corresponding full-order DG solution. 
For this test, the time step is chosen sufficiently small so that the temporal discretization error is negligible compared with the spatial discretization error; in particular, we take $N_t=20000$ and $\omega=1/2$. 
The POD basis is constructed from displacement snapshots with $N_s=1000$ and $\ell=5$, and the same reduced dimension is used for all mesh levels.

\Cref{fig:ex1:space:k1,fig:ex1:space:k2} present the maximum-in-time spatial errors in the $L_2$ and broken $H^1$ norms for $k=1$ and $k=2$, respectively. The observed convergence rates are consistent with the optimal DG estimates, namely $O(h^{k+1})$ in the $L_2$ norm and $O(h^k)$ in the broken $H^1$ norm. For $k=2$, a very slight deterioration of the $L_2$ convergence rate is observed on the finest mesh only for the DG-POD approximation with $W_u=M$. This minor deviation indicates that, at $N_{xy}=32$, the spatial discretization error has become comparable to the remaining temporal, POD truncation, and algebraic errors. Apart from this negligible effect, the DG-POD approximations reproduce essentially the same convergence behavior as the full-order DG solutions for all choices of $W_u$, confirming that the POD truncation error remains sufficiently small in this experiment.

\begin{figure}
\centering
\begin{subfigure}{0.48\textwidth}
    \centering
    \includegraphics[width=.9\linewidth]{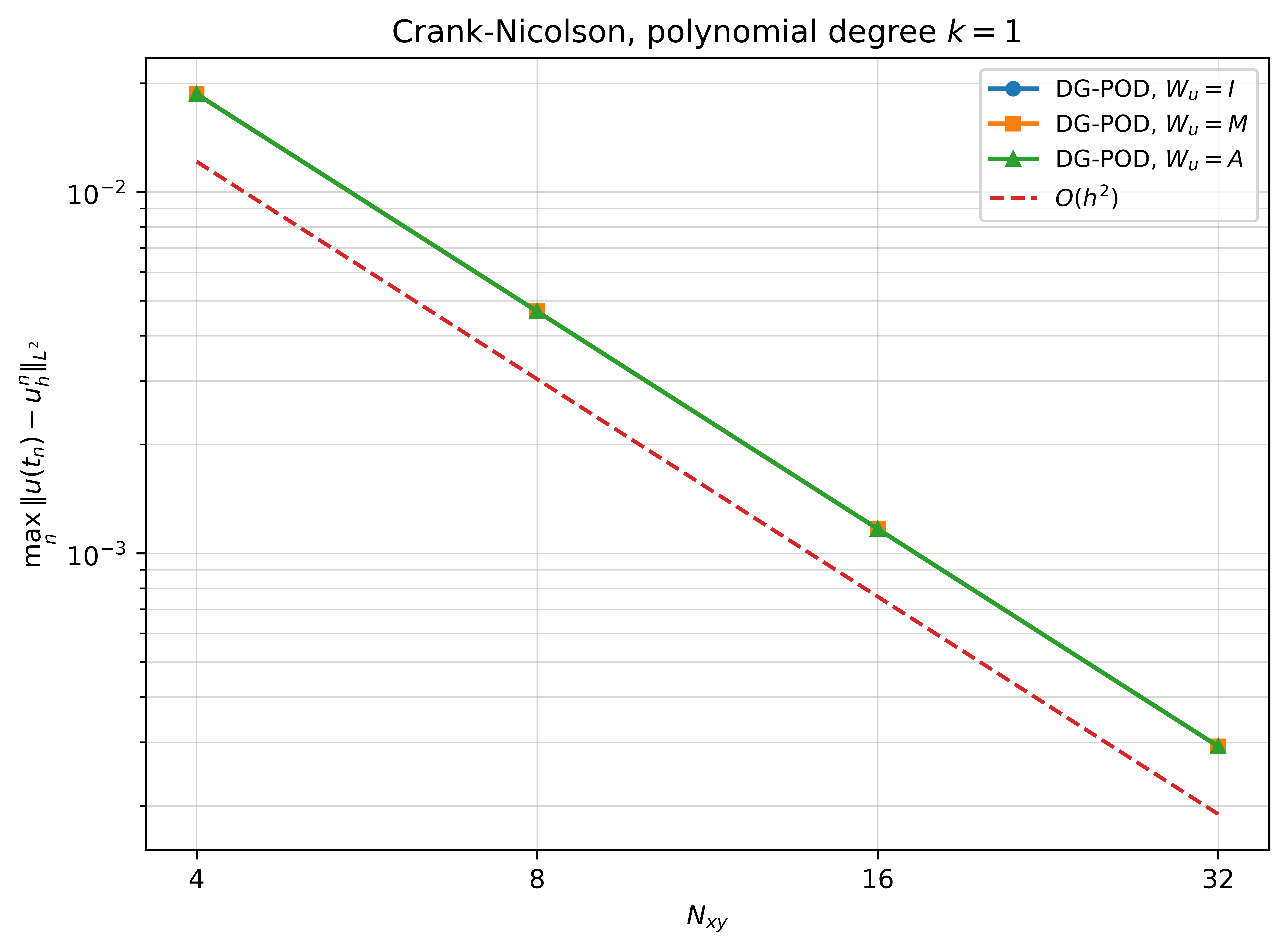}
    \caption{$L_2$ error.}
\end{subfigure}
\hfill
\begin{subfigure}{0.48\textwidth}
    \centering
    \includegraphics[width=.9\linewidth]{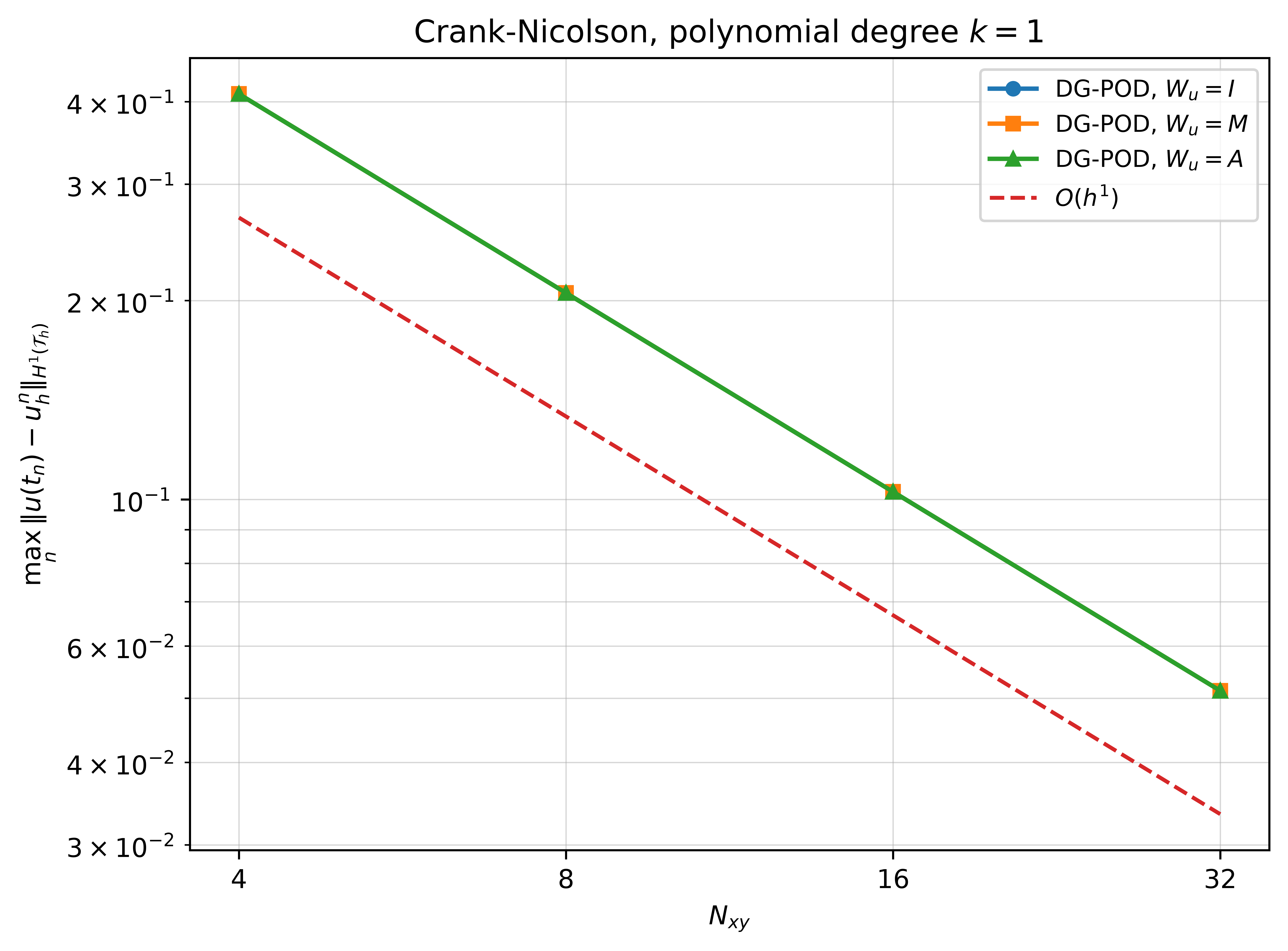}
    \caption{$H^1$ error.}
\end{subfigure}

\caption{Spatial convergence of the DG-POD approximation with the linear basis.}
\label{fig:ex1:space:k1}
\end{figure}
\begin{figure}
\centering
\begin{subfigure}{0.48\textwidth}
    \centering
    \includegraphics[width=.9\linewidth]{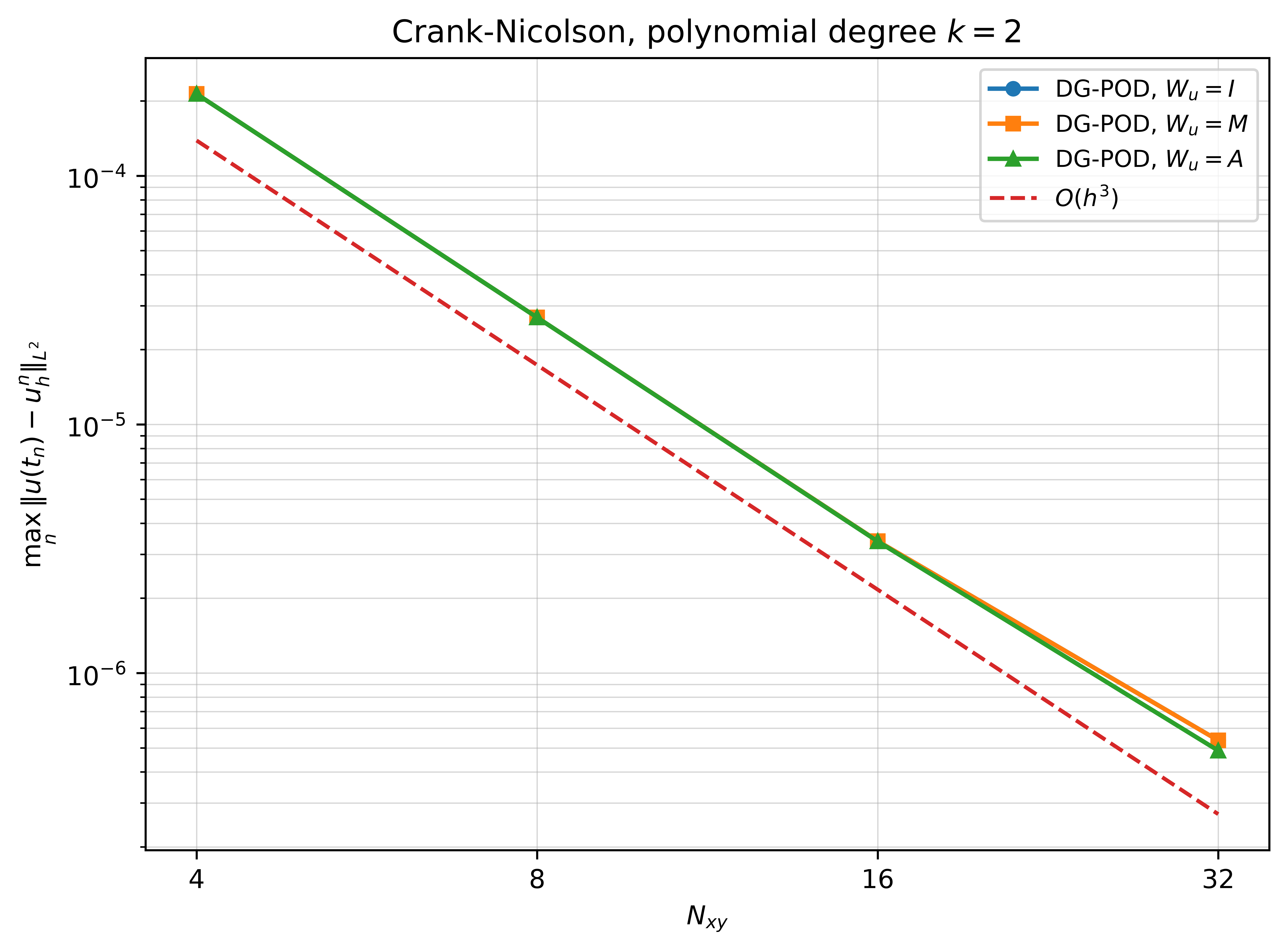}
    \caption{$L_2$ error.}
\end{subfigure}
\hfill
\begin{subfigure}{0.48\textwidth}
    \centering
    \includegraphics[width=.9\linewidth]{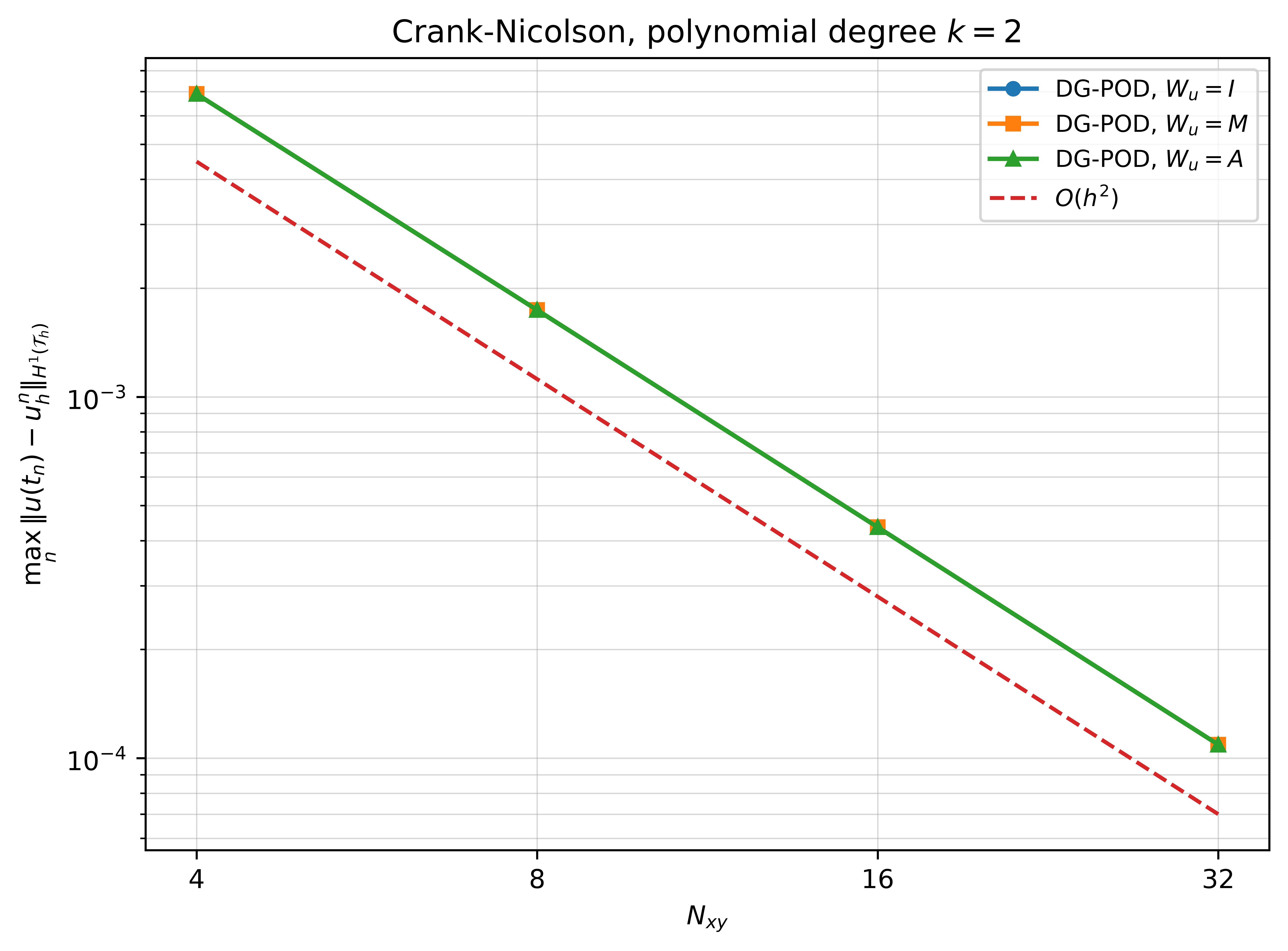}
    \caption{$H^1$ error.}
\end{subfigure}

\caption{Spatial convergence of the DG-POD approximation with the quadratic basis.}
\label{fig:ex1:space:k2}
\end{figure}

\begin{table}
\centering
\caption{Computational performance of the DG and DG-POD approximations for $k=1$ and $N_t=20000$.}
\label{tab:spatial:speedup:k1}
\small
\begin{tabular}{cc|cccc}
\hline
$N_{xy}$ & $W_u$ 
& Set-up[s]
& Solve[s] (speed-up)
& Set-up+Solve[s] (speed-up)
& Total[s] (speed-up)
\\
\hline
\multirow{4}{*}{4}
& FOM       & 0.013 & 0.510 (--)  & 0.522 (--)  & 0.522 (--) \\
& $I$   & 0.117 & 0.189 (2.69) & 0.306 (1.70) & 0.809 (0.65) \\
& $M$    & 0.079 & 0.161 (3.17) & 0.239 (2.18) & 0.745 (0.70) \\
& $A$  & 0.099 & 0.154 (3.31) & 0.253 (2.07) & 0.751 (0.70) \\
\hline
\multirow{4}{*}{8}
& FOM     & 0.008 & 1.615 (--)  & 1.623 (--)  & 1.623 (--) \\
& $I$    & 0.097 & 0.204 (7.93) & 0.300 (5.41) & 1.057 (1.54) \\
& $M$   & 0.107 & 0.193 (8.38) & 0.300 (5.41) & 0.992 (1.64) \\
& $A$  & 0.112 & 0.192 (8.40) & 0.304 (5.34) & 0.996 (1.63) \\
\hline
\multirow{4}{*}{16}
& FOM       & 0.023 & 7.466 (--)   & 7.489 (--)   & 7.489 (--) \\
& $I$    & 0.123 & 0.664 (11.25) & 0.786 (9.52) & 2.624 (2.85) \\
& $M$    & 0.122 & 0.562 (13.29) & 0.684 (10.94) & 2.358 (3.18) \\
& $A$  & 0.145 & 0.579 (12.88) & 0.725 (10.33) & 2.399 (3.12) \\
\hline
\multirow{4}{*}{32}
& FOM        & 0.093 & 53.950 (--)  & 54.044 (--)  & 54.044 (--) \\
& $I$    & 0.153 & 2.162 (24.96) & 2.315 (23.34) & 10.445 (5.17) \\
& $M$    & 0.228 & 2.095 (25.76) & 2.323 (23.27) & 8.543 (6.33) \\
& $A$ & 0.288 & 2.143 (25.18) & 2.431 (22.24) & 8.650 (6.25) \\
\hline
\end{tabular}
\end{table}

\begin{table}
\centering
\caption{Computational performance of the DG and DG-POD approximations for $k=2$ and $N_t=20000$.}
\label{tab:spatial:speedup:k2}
\small
\begin{tabular}{cc|cccc}
\hline
$N_{xy}$ & $W_u$ 
& Set-up[s]
& Solve[s] (speed-up)
& Set-up+Solve[s] (speed-up)
& Total[s] (speed-up)
\\
\hline
\multirow{4}{*}{4}
& FOM        & 0.008 & 0.744 (--)  & 0.752 (--)  & 0.752 (--) \\
& $I$    & 0.082 & 0.172 (4.32) & 0.254 (2.95) & 0.954 (0.79) \\
& $M$   & 0.085 & 0.161 (4.61) & 0.247 (3.04) & 0.884 (0.85) \\
& $A$  & 0.094 & 0.163 (4.55) & 0.258 (2.92) & 0.895 (0.84) \\
\hline
\multirow{4}{*}{8}
& FOM         & 0.014 & 3.381 (--)  & 3.395 (--)  & 3.395 (--) \\
& $I$    & 0.094 & 0.259 (13.03) & 0.353 (9.61) & 1.611 (2.11) \\
& $M$   & 0.099 & 0.251 (13.47) & 0.350 (9.70) & 1.627 (2.09) \\
& $A$  & 0.090 & 0.246 (13.74) & 0.336 (10.10) & 1.614 (2.10) \\
\hline
\multirow{4}{*}{16}
& FOM         & 0.068 & 27.677 (--)  & 27.745 (--)  & 27.745 (--) \\
& $I$    & 0.122 & 1.178 (23.50) & 1.300 (21.34) & 5.640 (4.92) \\
& $M$    & 0.137 & 0.977 (28.34) & 1.113 (24.92) & 5.670 (4.89) \\
& $A$  & 0.188 & 0.967 (28.61) & 1.155 (24.02) & 5.712 (4.86) \\
\hline
\multirow{4}{*}{32}
& FOM        & 0.539 & 208.874 (--) & 209.413 (--) & 209.413 (--) \\
& $I$    & 0.278 & 5.730 (36.45) & 6.008 (34.86) & 34.202 (6.12) \\
& $M$   & 0.411 & 4.674 (44.69) & 5.085 (41.19) & 29.475 (7.10) \\
& $A$  & 0.549 & 5.283 (39.54) & 5.832 (35.91) & 30.222 (6.93) \\
\hline
\end{tabular}
\end{table}

As shown in \Cref{tab:spatial:speedup:k1,tab:spatial:speedup:k2}, the DG-POD approximation significantly reduces the time-stepping cost, and the solve-time speed-up increases with spatial refinement. When the snapshot-generation cost is included, the total speed-up is relatively small on coarse meshes but increases for finer discretizations, as snapshot generation accounts for a smaller portion of the overall computational cost. The three choices of $W_u$ yield broadly comparable computational performance.

\paragraph{Temporal convergence}

We next investigate the temporal convergence of the proposed DG-POD approximation. To evaluate the temporal convergence order, the spatial discretization is fixed with $k=2$ and $N_{xy}=128$, while the number of time steps is successively refined. The reduced basis is constructed from $N_s=100$ snapshots with a fixed reduced dimension $\ell=20$. Both the backward Euler and Crank--Nicolson schemes are considered and compared for the three POD inner products $W_u=I$, $M$, and $A$. For the temporal convergence study, we consider only the maximum-in-time $L_2$ error. On the finer time grids, the broken $H^1$ error can be dominated by the fixed spatial discretization error and therefore no longer provides a clear measure of the temporal convergence rate.

\begin{figure}
\centering
\begin{subfigure}{0.48\textwidth}
    \centering
    \includegraphics[width=.9\linewidth]{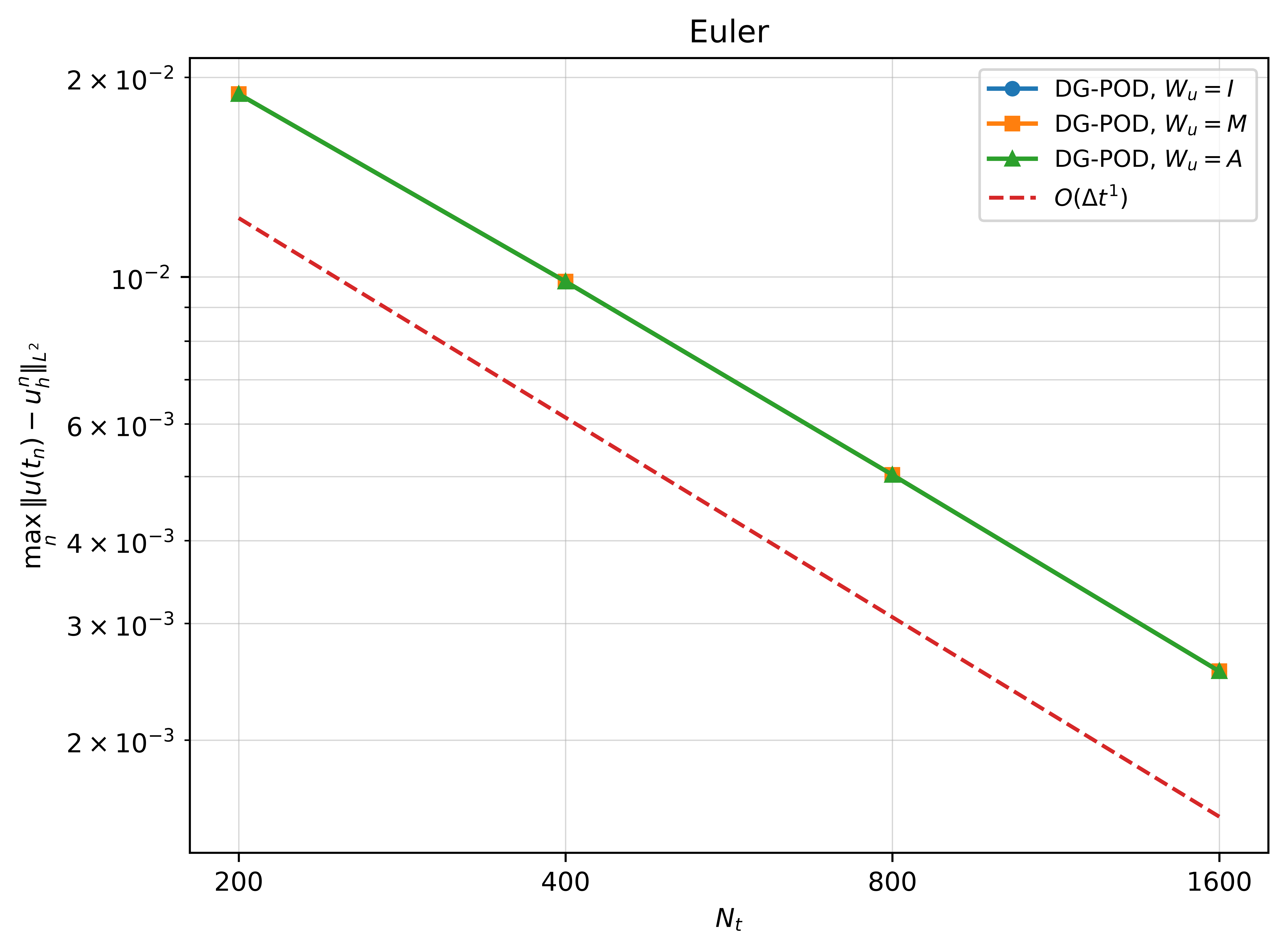}
    \caption{Backward Euler method.}
\end{subfigure}
\hfill
\begin{subfigure}{0.48\textwidth}
    \centering
    \includegraphics[width=.9\linewidth]{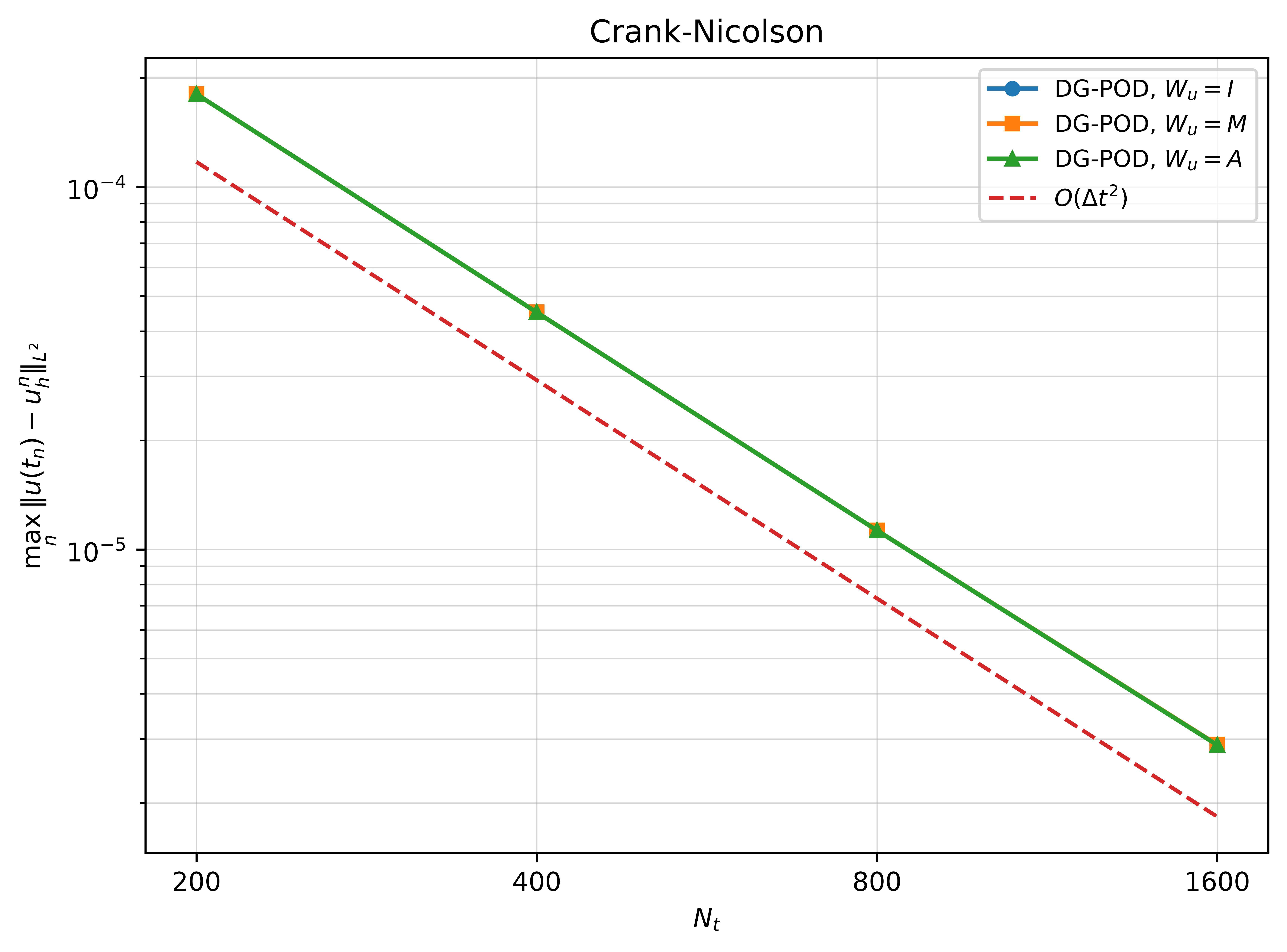}
    \caption{Crank--Nicolson Method.}
\end{subfigure}

\caption{Temporal convergence of the DG-POD approximation with $N_s=100$, $N_{xy}=128$ and $k=2$.}
\label{fig:ex1:time}
\end{figure}

\Cref{fig:ex1:time} presents the temporal convergence histories in the maximum-in-time $L_2$ norm. The backward Euler scheme exhibits the expected first-order convergence, whereas the Crank--Nicolson scheme achieves second-order convergence. The three choices of the POD inner product yield nearly identical temporal accuracy, indicating that the DG-POD approximation preserves the convergence properties of the underlying time-stepping schemes.

The computational results in \Cref{tab:temporal:speedup:euler,tab:temporal:speedup:cn} show that the DG-POD approximation reduces the time-stepping cost for both schemes. The solve-time speed-up is approximately 80 and remains nearly independent of the POD inner product. The speed-up based on the combined set-up and solve times increases with $N_t$, because the relative contribution of the fixed set-up cost decreases as the number of time steps increases. When the snapshot-generation cost is included, the total speed-up remains limited on coarse temporal grids because $N_s$ is comparable to $N_t$, making snapshot generation a major component of the overall simulation time. As the gap between $N_t$ and $N_s$ increases, the relative contribution of the snapshot-generation cost decreases, leading to a steadily larger total speed-up. In particular, the total speed-up reaches approximately 6 at $N_t=1600$.

The numerical results in \cref{ssec:numeric:smooth} confirm the \textit{a priori} error estimates with respect to both spatial and temporal discretizations. In particular, the expected optimal convergence orders are observed in the $L_2$ and broken $H^1$ norms for the spatial discretization, as well as in the $L_2$ norm for the backward Euler and Crank--Nicolson schemes. These convergence properties are preserved independently of the choice of POD inner product. Moreover, because the exact solution is smooth and separable, its dynamics can be accurately represented using only a small number of POD modes, provided that sufficiently many snapshots are available; for example, $\ell=5$ and $N_s=1000$ already yield an accurate approximation. In the next experiment, we investigate the influence of the POD dimension on a problem with highly oscillatory solution dynamics.

\begin{table}
\centering
\caption{Computational performance of the DG and DG-POD approximations using the backward Euler method.}
\label{tab:temporal:speedup:euler}
\small
\begin{tabular}{cc|cccc}
\hline
$N_t$ & $W_u$
& Set-up [s]
& Solve [s] (speed-up)
& Set-up+Solve [s] (speed-up)
& Total [s] (speed-up)
\\
\hline
\multirow{4}{*}{200}
& FOM & 26.90 & 62.47 (--) & 89.38 (--) & 89.38 (--) \\
& $I$ & 2.19 & 0.45 (137.54) & 2.64 (33.82) & 83.45 (1.07) \\
& $M$ & 2.41 & 0.42 (149.29) & 2.83 (31.63) & 83.63 (1.07) \\
& $A$ & 2.59 & 0.40 (154.40) & 3.00 (29.81) & 83.80 (1.07) \\
\hline
\multirow{4}{*}{400}
& FOM & 27.34 & 125.78 (--) & 153.12 (--) & 153.12 (--) \\
& $I$ & 2.09 & 0.88 (142.72) & 2.97 (51.57) & 83.77 (1.83) \\
& $M$ & 2.31 & 0.86 (145.85) & 3.17 (48.28) & 83.98 (1.82) \\
& $A$ & 2.56 & 0.77 (163.18) & 3.33 (45.95) & 84.14 (1.82) \\
\hline
\multirow{4}{*}{800}
& FOM & 27.37 & 245.79 (--) & 273.15 (--) & 273.15 (--) \\
& $I$ & 2.07 & 1.72 (142.81) & 3.80 (71.96) & 84.60 (3.23) \\
& $M$ & 2.33 & 1.74 (141.41) & 4.07 (67.20) & 84.87 (3.22) \\
& $A$ & 2.55 & 1.54 (159.28) & 4.09 (66.77) & 84.89 (3.22) \\
\hline
\multirow{4}{*}{1600}
& FOM & 27.32 & 531.31 (--) & 558.63 (--) & 558.63 (--) \\
& $I$ & 2.13 & 3.52 (151.06) & 5.65 (98.95) & 86.45 (6.46) \\
& $M$ & 2.43 & 3.50 (151.83) & 5.93 (94.27) & 86.73 (6.44) \\
& $A$ & 2.51 & 3.30 (160.97) & 5.81 (96.15) & 86.61 (6.45) \\
\hline
\end{tabular}
\end{table}

\begin{table}
\centering
\caption{Computational performance of the DG and DG-POD approximations using the Crank--Nicolson method.}
\label{tab:temporal:speedup:cn}
\small
\begin{tabular}{cc|cccc}
\hline
$N_t$ & $W_u$
& Set-up [s]
& Solve [s] (speed-up)
& Set-up+Solve [s] (speed-up)
& Total [s] (speed-up)
\\
\hline
\multirow{4}{*}{200}
& FOM & 26.94 & 60.63 (--) & 87.57 (--) & 87.57 (--) \\
& $I$ & 1.99 & 0.42 (143.01) & 2.41 (36.30) & 80.89 (1.08) \\
& $M$ & 2.33 & 0.42 (143.38) & 2.75 (31.82) & 81.23 (1.08) \\
& $A$ & 2.47 & 0.42 (146.00) & 2.89 (30.31) & 81.36 (1.08) \\
\hline
\multirow{4}{*}{400}
& FOM & 27.13 & 121.27 (--) & 148.41 (--) & 148.41 (--) \\
& $I$ & 2.08 & 0.81 (149.18) & 2.90 (51.26) & 81.37 (1.82) \\
& $M$ & 2.40 & 0.85 (142.57) & 3.25 (45.69) & 81.72 (1.82) \\
& $A$ & 2.51 & 0.85 (142.01) & 3.36 (44.13) & 81.84 (1.81) \\
\hline
\multirow{4}{*}{800}
& FOM & 27.11 & 261.41 (--) & 288.52 (--) & 288.52 (--) \\
& $I$ & 2.27 & 1.69 (154.64) & 3.96 (72.82) & 82.44 (3.50) \\
& $M$ & 2.40 & 1.73 (151.03) & 4.13 (69.87) & 82.60 (3.49) \\
& $A$ & 2.52 & 1.71 (152.46) & 4.23 (68.14) & 82.71 (3.49) \\
\hline
\multirow{4}{*}{1600}
& FOM & 27.25 & 482.60 (--) & 509.84 (--) & 509.84 (--) \\
& $I$ & 1.95 & 3.37 (143.05) & 5.32 (95.75) & 83.80 (6.08) \\
& $M$ & 2.21 & 3.32 (145.36) & 5.53 (92.18) & 84.01 (6.07) \\
& $A$ & 2.40 & 3.44 (140.49) & 5.83 (87.45) & 84.31 (6.05) \\
\hline
\end{tabular}
\end{table}

\subsection{Initial traction pulse problem}\label{ssec:traction}

To assess the proposed DG-POD method in a more realistic setting, we consider a two-dimensional viscoelastic wave problem. The relaxation parameters are adopted from representative viscoelastic material data~\cite{lin2009viscoelastic}, while the elastic constants are selected for numerical illustration. The material parameters are
\[
\mu = \SI{0.455}{\mega\pascal},\q
\nu = 0.45,\q
\rho = \SI{965}{\kilogram\per\meter\cubed}, \q
\varphi_0=0.89,\q
\varphi_1=0.08,\q
\varphi_2=0.03,\q
\tau_1=0.165,\q
\tau_2=5.
\]
Under the plane-strain assumption, the first Lam\'e parameter is
\[
\lambda
=
\frac{2\mu\nu}{1-2\nu}
=
\SI{4.095}{\mega\pascal}.
\]
The fourth-order tensor is defined by
\[
D_{ijkl}
=
\lambda\,\delta_{ij}\delta_{kl}
+
\mu
\left(
\delta_{ik}\delta_{jl}
+
\delta_{il}\delta_{jk}
\right).
\]
We take $\Omega=(0,1)^2$, $T=5$ and impose homogeneous Dirichlet boundary conditions on all boundaries except the left boundary
\[
\Gamma_N=\{(x,y)\in\partial\Omega:\ x=0\}.
\]
The body force and initial data are set to zero:
$\boldsymbol f=\boldsymbol u_0=\boldsymbol w_0=\boldsymbol 0$.

A traction force is applied only on $\Gamma_N$ at the initial time step:
\[
\boldsymbol g_N(0,y,0)
=
g_0
\begin{pmatrix}
\sin(\pi y)\\
0
\end{pmatrix},
\qquad y\in(0,1),\qq\text{whereas}\qq 
\boldsymbol g_N(0,y,t)=\boldsymbol 0,\qquad t>0.
\]
Thus, after the initial traction pulse, the subsequent dynamics are governed solely by wave propagation.

Since no analytical solution is available, the DG-FOM solution with $N_t=8000$ is used as the reference. All computations use the same spatial discretization with $N_{xy}=64$ and $k=2$, together with the Crank--Nicolson method.
The POD basis is constructed from a separate DG-FOM simulation using
$N_s=1000$ snapshots. The DG-POD simulations employ the same spatial
discretization and Crank--Nicolson time-stepping scheme as the reference
DG-FOM solution. With $W_u=I$, the number of POD basis functions is varied as
\[
\ell=10,\ 20,\ 50,\ 100,\ 200.
\]

The reduced-order performance is evaluated from two complementary viewpoints. First, the transient response is assessed by monitoring the $x$-displacement (horizontal displacement) $u_1(\boldsymbol{x}_p,t)$ at the observation point
\[
\boldsymbol{x}_p=(0.203,0.497),
\]
which is located inside the computational domain. This comparison provides a quantitative measure of how accurately the reduced-order model reproduces the wave propagation. Second, the computational efficiency of the proposed method is assessed by comparing the setup time, solution time, and total simulation time with those of the full-order DG solver. We also report the relative POD tail defined by
\[
\eta_\ell=\left(\frac{\sum_{j=\ell+1}^{m}\sigma_j^2}{\sum_{j=1}^{m}\sigma_j^2}\right)^{1/2},
\]
which measures the relative magnitude of the discarded snapshot components.

\paragraph{Influence of the reduced basis dimension.}

\Cref{fig:ex2:wave} compares the displacement histories at the observation point $\bs{x}_p=(0.203,0.497)$ over the time domain $0\leq t\leq 1$ for different reduced dimensions.
With 10 POD basis functions, the DG-POD approximation reproduces the main displacement response generated by the applied traction, although some of the higher-frequency oscillations are smoothed. As the reduced dimension increases, the oscillatory structure is progressively recovered, and the DG-POD solution approaches the reference DG-FOM history. In particular, 50 POD basis functions provide a close approximation of both the global waveform and the principal transient oscillations over this time interval. This behavior indicates that the dominant displacement response is captured by a relatively small number of POD modes, whereas additional modes are required to resolve the finer-scale wave components. Nevertheless, as shown in \Cref{fig:ex2:wave:long}, all DG-POD solutions reproduce the long-time relaxation behavior, indicating that even low-dimensional reduced spaces capture the characteristic viscoelastic material response.

\begin{figure}
    \centering
    \includegraphics[width=\linewidth,
    trim={1.9cm 1.0cm 0.0cm 2.0cm},
    clip]{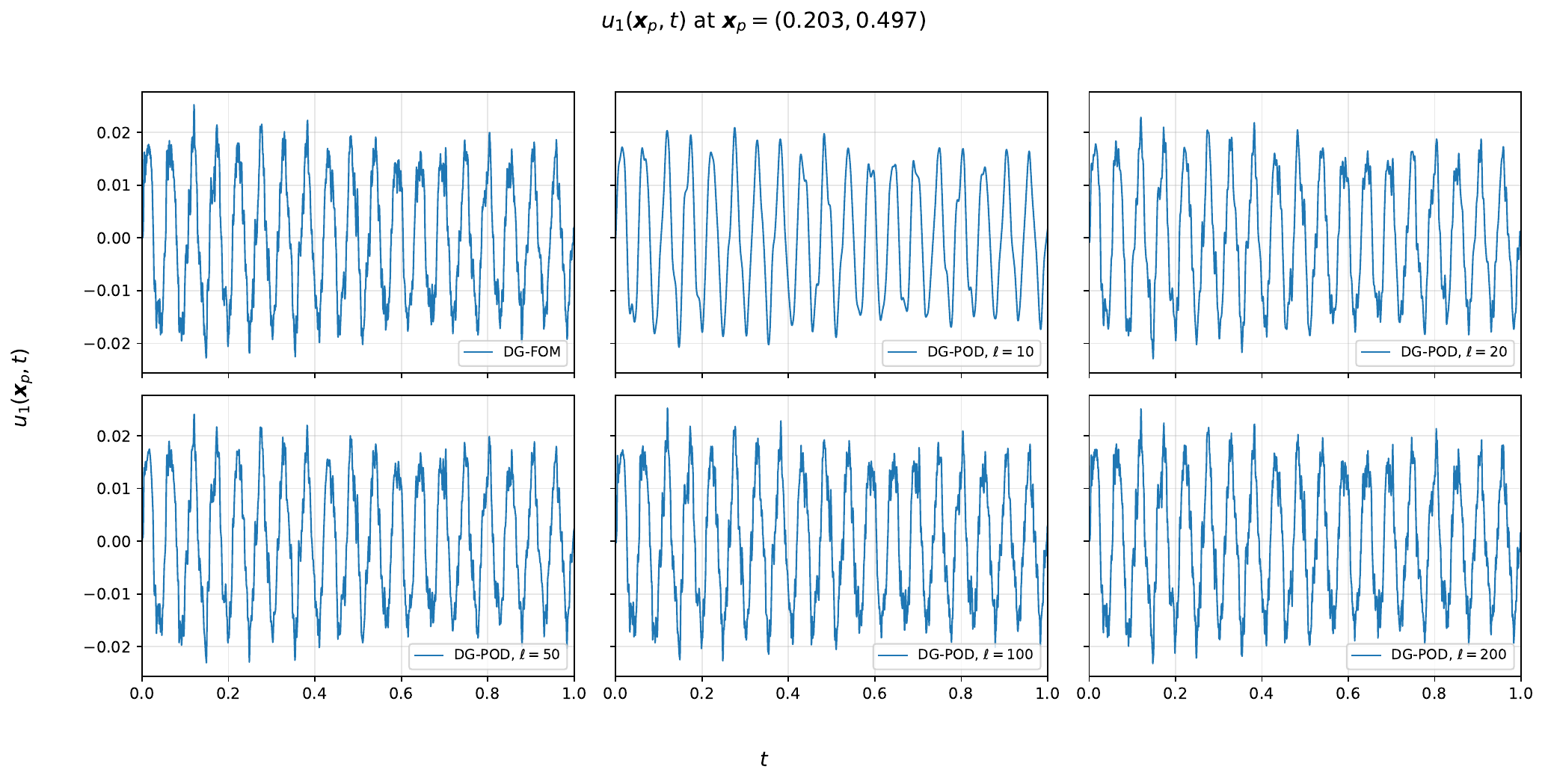}
\caption{Displacement histories at
    $\boldsymbol{x}_p=(0.203,0.497)$ over the time domain
    $0\leq t\leq 1$ for different POD dimensions.}
\label{fig:ex2:wave}
\end{figure}
\begin{figure}
    \centering
    \includegraphics[width=\linewidth,
    trim={1.9cm 1.0cm 0.0cm 2.0cm},
    clip]{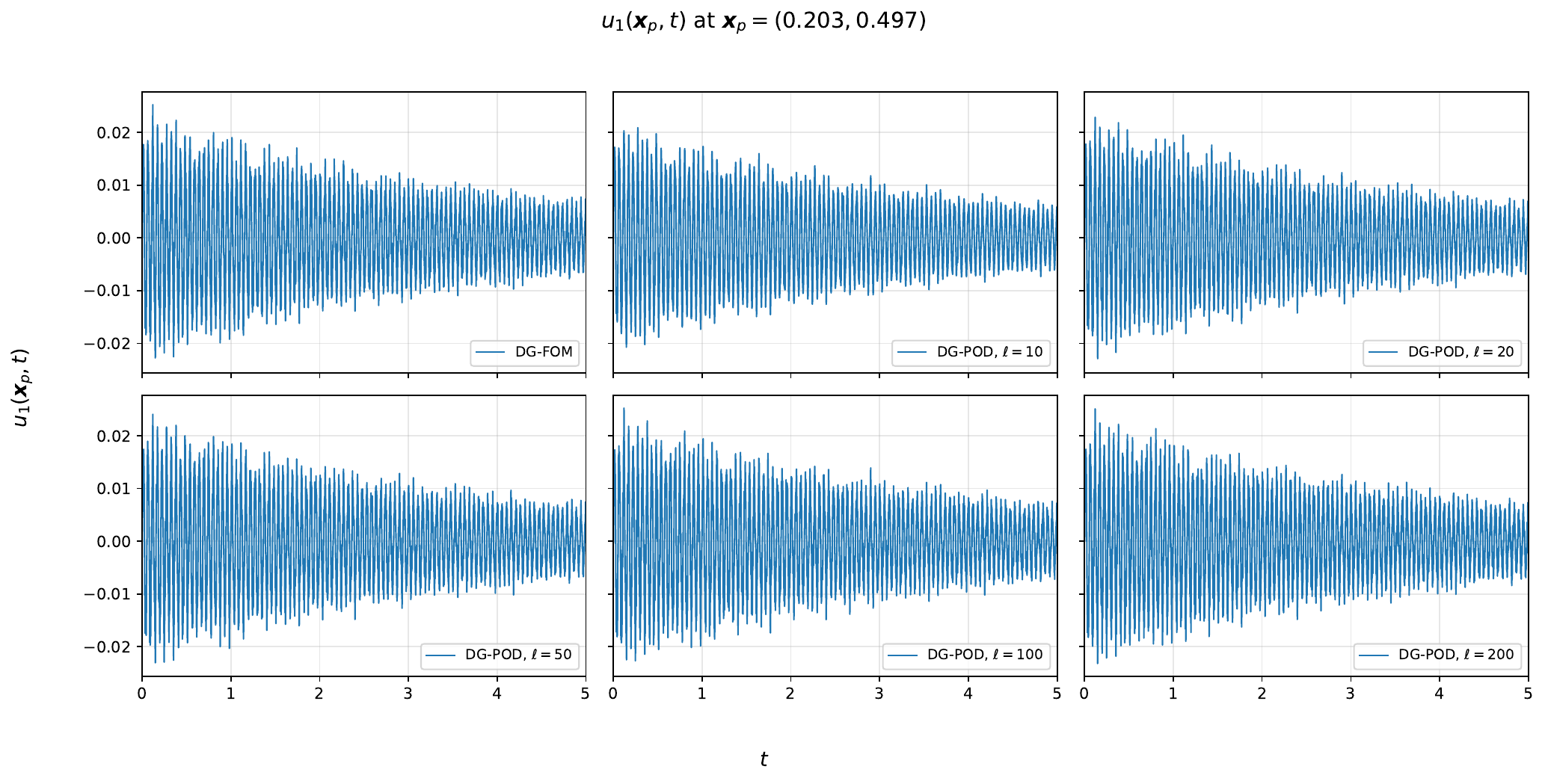}
\caption{Displacement histories at
    $\boldsymbol{x}_p=(0.203,0.497)$ over the time domain
    $0\leq t\leq 5$ for different POD dimensions.}
\label{fig:ex2:wave:long}
\end{figure}

\Cref{tab:ex2:performance} further quantifies the computational performance of the proposed DG-POD method. As the reduced dimension increases from $\ell=10$ to $\ell=200$, the relative POD tail decreases from $1.46\times10^{-1}$ to $6.52\times10^{-3}$, indicating that progressively less snapshot energy is discarded. Since the present problem exhibits a highly oscillatory transient response, additional POD modes are required to represent the finer temporal features and reduce the truncation error. This trend is consistent with the increasingly accurate recovery of the oscillatory response observed in \Cref{fig:ex2:wave}.

From a computational perspective, increasing the number of POD basis functions raises both the reduced-system setup and solution costs because larger reduced operators must be assembled and solved. Nevertheless, even for $\ell=200$, the DG-POD solving phase is approximately $5.19\times10^{3}$ times faster than that of the reference DG-FOM computation. When the snapshot-generation cost is also included, the total computational time is approximately $78$s, compared with $582$s for the reference DG-FOM simulation, corresponding to an overall speed-up of about $7.5$. These results demonstrate that the DG-POD method remains computationally advantageous even after accounting for the offline cost of generating the snapshots and constructing the reduced basis.

\begin{table}
\centering
\caption{Computational performance of the DG and DG-POD approximations for different POD dimensions.}
\label{tab:ex2:performance}
\small
\begin{tabular}{cc|c|cccc}
\hline
Method & $\ell$ & $\eta_\ell$
& Set-up [s]
& Solve [s] (speed-up)
& Set-up+Solve [s] (speed-up)
& Total [s] (speed-up)
\\
\hline
FOM
& -- & --
& 3.85
& 577.05 (--)
& 580.91 (--)
& 582.10 (--)
\\
\hline
POD
& 10 & 1.46e-01
& 0.68
& 0.03 (2.04e+04)
& 0.71 (8.24e+02)
& 76.86 (7.57)
\\
\hline
POD
& 20 & 9.75e-02
& 0.66
& 0.03 (2.15e+04)
& 0.69 (8.40e+02)
& 76.80 (7.58)
\\
\hline
POD
& 50 & 4.82e-02
& 0.91
& 0.04 (1.52e+04)
& 0.94 (6.15e+02)
& 77.09 (7.55)
\\
\hline
POD
& 100 & 2.25e-02
& 1.23
& 0.05 (1.16e+04)
& 1.28 (4.54e+02)
& 77.42 (7.52)
\\
\hline
POD
& 200 & 6.52e-03
& 1.80
& 0.11 (5.19e+03)
& 1.91 (3.05e+02)
& 78.12 (7.45)
\\
\hline
\end{tabular}
\end{table}
\section{Conclusion}\label{sec:conclusion}
In this paper, we proposed and analyzed a discontinuous Galerkin proper orthogonal decomposition (DG-POD) method for dynamic linear viscoelasticity based on an internal-variable formulation. We established the well-posedness and derived \textit{a priori} spatial error estimates for the semi-discrete DG problem. For the fully discrete DG-POD approximation, we presented an error decomposition that separates the contributions of the spatial discretization, temporal discretization, and POD truncation, under appropriate stability and projection assumptions for the reduced evolution. These results provide a theoretical basis for the proposed DG-POD framework.

Numerical experiments confirmed the accuracy, robustness, and computational efficiency of the proposed DG-POD framework across both smooth benchmark problems and transient wave propagation. The reduced model preserved the convergence behavior of the underlying DG discretization and reproduced the dominant dynamics of the full-order system with lower computational cost. Moreover, the results illustrated the expected accuracy--complexity trade-off: low-dimensional reduced spaces captured the large-scale response, whereas progressively richer POD spaces were required to resolve finer transient oscillations. These findings indicate that the proposed methodology provides a broadly applicable reduced-order framework for time-dependent viscoelastic systems. Future work will focus on extensions to parameter-dependent, nonlinear, and more general history-dependent continuum models, together with adaptive snapshot selection, basis enrichment, and hyper-reduction strategies for large-scale repeated simulations.

\section*{Acknowledgments}
This work was partially supported by Chonnam National University(No. 2025-1877-01) and by Global-Learning \& Academic research institution for Master’s $\cdot$ PhD students, and Postdocs(LAMP) Program of the National Research Foundation of Korea(NRF) grant funded by the Ministry of Education(No. RS-2024-00442775). Also, it was supported by the National Research Foundation of Korea (NRF) grant funded by the
Korea government (MSIT) (No. RS-2023-NR076790).

\section*{Conflict of interest}
The authors declare that they have no conflict of interest.

\section*{Data availability}
The source code and data used to generate the numerical results presented in this work are publicly available at the author's GitHub repository:
\url{https://github.com/Yongseok7717/pod_visco}.

\section*{Declaration of generative AI use}
The authors
declare they have not used Artificial Intelligence (AI) tools in the creation of this article.
\printcredits

\bibliographystyle{cas-model2-names}

\bibliography{refs}



\end{document}